\documentclass[a4paper,12pt]{amsart}

\usepackage{amsmath,amssymb,amsthm,graphicx,pstricks,comment}
\usepackage[all]{xy}

\usepackage{lscape}
\usepackage{hyperref}
\usepackage{tikz}
\usetikzlibrary{arrows}
\usetikzlibrary{snakes}
\usetikzlibrary{shapes}

\author{Claire Amiot}
\address{Institut Fourier, Universit\'e Joseph Fourier, 100 rue des maths, 38402 Saint Martin d'H\`eres, France}
\email{claire.amiot@ujf-grenoble.fr}\thanks{All authors were supported by the project 196600/V30 from the Norwegian Research Council.}\thanks{The first author is partially supported by the ANR project ANR-09-BLAN-0039-02.}
\author{Osamu Iyama}
\address{Graduate School of Mathematics, Nagoya University}
\email{iyama@math.nagoya-u.ac.jp}
\thanks{The second author was supported by JSPS Grant-in-Aid for Scientific Research 21740010, 21340003, 20244001 and 22224001}

\author{Idun Reiten}
\address{Insitutt for matematiske fag,
Norges Teknisk-Naturvitenskapelige Universitet,
N-7491 Trondheim, Norway}
\email{idunr@math.ntnu.no}

\dedicatory{Dedicated to Ragnar-Olaf Buchweitz on the occasion of his
sixtieth birthday}

\numberwithin{equation}{subsection}

\newcommand{\Hom}{{\sf Hom }}\newcommand{\RHom}{\mathbf{R}{\sf Hom }}
\newcommand{\End}{{\sf End }}
\newcommand{\Ext}{{\sf Ext }}

\renewcommand{\mod}{{\sf mod \hspace{.02in}  }}
\newcommand{\fd}{{\sf fd \hspace{.02in}  }}
\newcommand{\Mod}{{\sf Mod \hspace{.02in} }}

\newcommand{\per}{{\sf per \hspace{.02in}  }}
\newcommand{\proj}{{\sf proj \hspace{.02in} }}

\newcommand{\add}{{\sf add \hspace{.02in} }}
\newcommand{\gr}{{\sf gr \hspace{.02in} }}
\newcommand{\grproj}{{\sf grproj \hspace{.02in} }}
\newcommand{\Gr}{{\sf Gr \hspace{.02in} }}
\newcommand{\thick}{{\sf thick  \hspace{.02in} }}
\newcommand{\ten}{\otimes}
\newcommand{\lten}{\overset{\mathbf{L}}{\ten}}

\newcommand{\Talg}{{\sf T}}
\newcommand{\SL}{{\sf SL}}
\newcommand{\gldim}{{\sf gl.dim \hspace{.02in}}}
\newcommand{\injdim}{{\sf inj.dim}}
\newcommand{\projdim}{{\sf proj.dim}}

\newcommand{\A}{\underline{A}}
\newcommand{\B}{\underline{B}}
\newcommand{\Aa}{\mathcal{A}}

\newcommand{\Cc}{\mathcal{C}}
\newcommand{\Dd}{\mathcal{D}}

\newcommand{\Kk}{\mathcal{K}}
\newcommand{\Tt}{\mathcal{T}}

\newcommand{\Vv}{\mathcal{V}}

\newcommand{\SSS}{\mathbb{S}}

\newcommand{\ZZ}{\mathbb{Z}}
\newcommand{\CM}{{\sf CM \hspace{.02in}}}
\newcommand{\Db}{\mathcal{D}^{\rm b}}
\newcommand{\BPi}{{\mathbf\Pi}}

\newcommand{\bsm}{\begin{smallmatrix}}
\newcommand{\esm}{\end{smallmatrix}}

\newtheorem{thm}{Th{\'e}or{\`e}me}[section]
\newtheorem{thma}{Theorem}[section]
\newtheorem*{thm*}{Th{\'e}or{\`e}me}
\newtheorem{lema}[thma]{Lemma}
\newtheorem*{lem*}{Lemme}

\newtheorem{cora}[thma]{Corollary}
\newtheorem*{prop*}{Proposition}
\newtheorem{prop}[thma]{Proposition}
\theoremstyle{remark}

\newtheorem{rema}[thma]{Remark}

\newtheorem{exa}[thm]{Example}
\theoremstyle{definition}

\newtheorem{dfa}[thma]{Definition}

\begin{document}
\title{Stable categories of Cohen-Macaulay modules and cluster categories}
\begin{abstract}
By Auslander's algebraic McKay correspondence,
the stable category of Cohen-Macaulay modules over a simple singularity
is triangle equivalent to the $1$-cluster category of the path algebra of a Dynkin quiver (i.e. the orbit category of the derived category  by the action of  the Auslander-Reiten translation).
In this paper we give a systematic method to construct a similar type of triangle equivalence between the stable category
of Cohen-Macaulay modules over a Gorenstein isolated singularity $R$ and the generalized (higher) cluster category of a finite dimensional
algebra $\Lambda$. The key role is played by a bimodule Calabi-Yau algebra, which is the
higher Auslander algebra of $R$ as well as the higher preprojective algebra of an extension of $\Lambda$.
As a byproduct, we give a triangle equivalence between the stable category of graded Cohen-Macaulay
$R$-modules and the derived category of $\Lambda$. Our main results apply in particular to a class of cyclic quotient singularities and to certain toric affine threefolds associated with dimer models. 
\end{abstract}

\maketitle
\tableofcontents

\section*{Introduction}

There has recently been a lot of interest centered around $\Hom$-finite triangulated
Calabi-Yau categories over a field $k$, especially in dimension two.
The work on 2-Calabi-Yau categories was originally motivated by trying to categorify the
ingredients in the definition of the cluster algebras introduced by Fomin and Zelevinsky~\cite{FZ02}.
It started in~\cite{BMR+06} through the cluster categories together with a special class of objects called cluster tilting objects,
and in~\cite{GLS06,BIRS09,GLS07,IO09} through the investigation of preprojective algebras and their higher analogs. 

The generalized $n$-cluster categories associated with finite dimensional algebras of global dimension at most $n$ were introduced in~\cite{Ami09,Guo10}.
In these categories, special objects called $n$-cluster tilting play an important role.
The cluster categories are a special case of the generalized 2-cluster categories, and the 2-cluster tilting objects are then the cluster tilting objects.
The generalized $n$-cluster categories can be considered to be the canonical ones among $n$-Calabi-Yau triangulated categories having $n$-cluster tilting objects.

On the other hand, a well-known example of Calabi-Yau triangulated categories was given in old work by Auslander~\cite{Aus78},
where the stable category of (maximal) Cohen-Macaulay modules over commutative isolated $d$-dimensional local Gorenstein singularities are shown to be $(d-1)$-Calabi-Yau.
Recently they are studied from the viewpoint of higher analog of Auslander-Reiten theory,
and the existence of $(d-1)$-cluster tilting objects is shown for quotient singularities in~\cite{Iya07a} and for some three dimensional hypersurface singularities in~\cite{BIKR08}.
They are further investigated in~\cite{IY08,KR08,KMV08}.

It is of interest to understand the relationship between these two classes of Calabi-Yau triangulated categories, i.e. the stable categories of Cohen-Macaulay modules and the generalized $n$-cluster categories.
A well-known example is given by Kleininan singularities.
They are given as hypersurfaces $R=k[x,y,z]/(f)$ as well as invariant subrings $R=S^G$ of $G$,
where $S=k[X,Y]$ is a polynomial algebra over an algebraically closed field $k$ of characteristic zero
and $G$ is a finite subgroup of $\SL_2(k)$. The correspondence between $f$ and $G$ is given as follows.
\[\begin{array}{|c||c|c|c|c|c|}\hline
\mbox{type}&A_n&D_n&E_6&E_7&E_8\\ \hline
f&x^{n+1}+yz&x^{n-1}+xy^2+z^2&x^4+y^3+z^2&x^3y+y^3+z^2&x^5+y^3+z^2\\ \hline
G&\begin{smallmatrix}\quad \\ \quad \\ \mbox{cyclic}\\ \quad\end{smallmatrix}&\begin{smallmatrix}\mbox{binary}\\ \mbox{dihedral}\end{smallmatrix}&
\begin{smallmatrix}\mbox{binary}\\ \mbox{tetrahedral}\end{smallmatrix}&\begin{smallmatrix}\mbox{binary}\\ \mbox{octahedral} \end{smallmatrix}&
\begin{smallmatrix}\mbox{binary}\\ \mbox{icosahedral}\end{smallmatrix} \\ \hline
\end{array}\]
In this case the stable category $\underline{\CM}(R)$ is equivalent to the mesh category $M(\overline{Q})$ of the Auslander-Reiten quiver of $\underline{\CM}(R)$, which is the double $\overline{Q}$ of a Dynkin quiver $Q$ \cite{Rei87,RV89}. 
On the other hand, $M(\overline{Q})$ is equivalent to the 1-cluster category $\Cc_1(kQ)$ of $Q$, i.e. the orbit category $\Db(kQ)/\tau$ of the derived category $\Db(kQ)$ by the action of $\tau$.
Hence we can deduce an equivalence 
\begin{equation}\label{motivation}
\underline{\CM}(R)\simeq\Cc_1(kQ),
\end{equation}
which is in fact a triangle equivalence (see Remark~\ref{in fact triangle}).
One of the aims of this paper is to prove this type of triangle equivalences for a more general class of quotient singularities.

Some crucial observations in the above setting are the following, where $\hat{R}$ and $\hat{S}$ are the completions of $R$ and $S$ at the origin respectively:
\begin{itemize}
\item \cite{Her78,Aus86} We have $\CM(R)=\add S$ and $\CM(\hat{R})=\add\hat{S}$. In particular $\hat{R}$ is representation-finite in the sense that there are only finitely many indecomposable Cohen-Macaulay modules.
\item \cite{Aus86} The Auslander algebra $\End_{\hat{R}}(\hat{S})$ (respectively, $\End_R(S)$) is isomorphic to the skew group algebra $\hat{S}*G$ (respectively, $S*G$). In particular, the AR quiver of $\CM(\hat{R})$ is isomorphic to the McKay quiver of $G$, which is the double of an extended Dynkin quiver $\widetilde{Q}$. (Note that in this case one has a triangle equivalence between the stable categories $\underline{\CM}(R)\simeq\underline{\CM}(\hat{R})$)
\item \cite{Rei87,RV89,BSW10} $S*G$ is Morita-equivalent to the preprojective algebra $\Pi$ of $\widetilde{Q}$. Hence $k\widetilde{Q}$ is the degree zero part of a certain grading of $\Pi$.
\end{itemize}
In particular the equivalence \eqref{motivation} is a direct consequence of the above observations. Also we have the following bridge between $R$ and $kQ$, where $e$ is the idempotent of $\End_R(S)\simeq S*G$ corresponding to the summand $R$ of $S$:
\[\xymatrix@C=8em{
R\ar@<.3ex>[r]^-{\mbox{\scriptsize Auslander algebra}}&S*G\stackrel{{\rm Morita}}{\sim}\Pi\ar@<.3ex>[r]^-{\mbox{\scriptsize degree 0 part}}\ar@<.3ex>[l]^-{e(-)e}&k\widetilde{Q}\ar@<.3ex>[r]^-{-/\langle e\rangle}\ar@<.3ex>[l]^-{\mbox{\scriptsize preprojective algebra}}&kQ
}\]
We will deal with the more general class of quotient singularities $S^G$,
where $S=k[x_1,\ldots,x_d]$ and $G$ is a finite cyclic subgroup of the special linear subgroup $\SL_d(k)$ with additional conditions, where no $g\neq1$ has eigenvalue $1$.
We will construct in Theorem~\ref{main result for singularities} a triangle equivalence
\begin{equation}\label{eq.equivalence}\underline{\CM}(S^G)\simeq\Cc_{d-1}(\underline{A})\end{equation}
for the generalized $(d-1)$-cluster category $\Cc_{d-1}(\underline{A})$ of some algebra $\underline{A}$ of global dimension at most $d-1$, which we describe.
This is shown as a special case of our main Theorem~\ref{main diagram}.
There we start from a bimodule $d$-Calabi-Yau graded algebra $B$ of Gorenstein parameter $1$ (e.g. $B$ is the skew group algebra $S*G$ when we deal with quotient singularities with additional conditions).
For an idempotent $e$ satisfying certain axioms, we have a similar picture as above:
\[\xymatrix@C=4cm{
eBe\ar@<.3ex>[r]^-{\mbox{\scriptsize $(d-1)$-Auslander algebra}}&B\ar@<.3ex>[r]^-{\mbox{\scriptsize degree 0 part}}\ar@<.3ex>[l]^-{e(-)e}&B_0\ar@<.3ex>[r]^-{-/\langle e\rangle}\ar@<.3ex>[l]^-{\mbox{\scriptsize $d$-preprojective algebra}}&B_0/\langle e\rangle
}\]
Our main result asserts that there exists a triangle equivalence
\[\underline{\CM}(eBe)\simeq\Cc_{d-1}(B_0/\langle e\rangle).\]
In addition to the quotient singularities already mentioned, this also applies to some examples coming from dimer models.

The main step of the proof consists of constructing a triangle equivalence 
$$\underline{\CM}^{\ZZ}(eBe)\simeq \Db(B_0/\langle e\rangle)$$
where $\CM^{\ZZ}(eBe)$ is the category of graded Cohen-Macaulay $eBe$-modules.
This intermediate result in the case where $B=S*G$ recovers a result due to Kajiura-Saito-Takahashi~\cite{KST07} and Lenzing-de la Pe\~na~\cite{LP06} for $d=2$ and due to Ueda~\cite{Ued08} for any $d$ and $G$ cyclic.
Moreover the triangle equivalence \eqref{eq.equivalence} was already shown in \cite{KR08} for the case $d=3$ and $G={\rm diag}(\omega,\omega,\omega)$ where $\omega$ is a primitive third root of unity.
It would be interesting to generalize our result to non-cyclic quotient singularities. This could then be regarded as an analog of
a triangle equivalence $\underline{\CM}^{\ZZ}(S^G)\simeq \Db(\Lambda)$ for some finite dimensional algebra $\Lambda $ given in~\cite{IT10}.

Results of a similar flavor have been shown in previous papers.
In~\cite{Ami09,ART11,AIRT10}, it was shown that the 2-Calabi-Yau categories
$\Cc_w$ associated with elements $w$ in Coxeter groups in \cite{BIRS09} are triangle equivalent to generalized $2$-cluster categories $\Cc_{2}(\underline{A})$ for some algebras $\underline{A}$ of global dimension at most two.
In~\cite{IO09}, it was shown that the stable categories of modules over $d$-preprojective
algebras of $(d-1)$-representation-finite algebras are triangle equivalent to
generalized $d$-cluster categories of stable $(d-1)$-Auslander algebras.
We were able to use some of the ideas in these papers for $d\ge 2$.

We refer to \cite{TV10} for similar independent results based on the language of quivers with potential. We thank Michel Van den Bergh for informing us about his work with Thanhoffer de V\"olcsey.

Some results in this paper were presented at a workshop in Oberwolfach (May 2010) \cite{AIR10}, Tokyo (August 2010), Banff (September 2010), Bielefeld (May 2011), Paris (June 2011), Shanghai (September 2011), Trondheim (March 2012), Banff (May 2012) and Guanajuato (May 2012).

\medskip
In section~1 we give some background material on  $n$-cluster tilting subcategories in 
$n$-Calabi-Yau categories and on generalized $n$-cluster categories.
Let $B$ be a bimodule $d$-Calabi-Yau algebra (see Definition~\ref{definition of CY algebra}) with an idempotent $e$, and let $C=eBe$.
In section~2, under certain conditions on $B$ and $e$, we show that $C$ is an Iwanaga-Gorenstein algebra
(see Definition~\ref{def of IG}),
and that $Be$ is a $(d-1)$-cluster tilting object in the category $\CM(C)$ of Cohen-Macaulay $C$-modules.
In section~3, which is independent of section 2, we assume that $B=\bigoplus_{\ell\ge0}B_\ell$ is graded, and give sufficient conditions for $B$ to be the $d$-preprojective algebra of $A=B_0$.
In particular $A$ is a $(d-1)$-representation-infinite algebra in the sense of \cite{HIO}
and a quasi extremely-Fano algebra in the sense of~\cite{MM10}.
In section~4, we use results from sections 2 and 3 to prove our main result, which gives sufficient conditions for the stable category $\underline{\CM}(C)$ to be triangle equivalent to
a generalized $(d-1)$-cluster category.
The application to $C$ being an invariant ring is given in section~5.
In section~6 we apply our main result to Jacobian algebras constructed from dimer models on the torus.

\subsection*{Notation}
Let $k$ be a field. We denote by $D=\Hom_k(-,k)$ the $k$-dual.
All modules are right modules.

For a $k$-algebra $A$, we denote by $\Mod A$ the category of $A$-modules, by $\mod A$ the category of finitely generated $A$-modules and by $\fd A$ the category of finite dimensional $A$-modules.
We let $\ten:=\ten_k$ and $A^{\rm e}:=A^{\rm op}\ten A$.
For a $\ZZ$-graded $k$-algebra $B$, we denote by $\Gr B$ the category of all $\ZZ$-graded $B$-modules, by $\gr B$ the category of finitely generated $\ZZ$-graded $B$-modules and by $\grproj B$ the category of finitely generated $\ZZ$-graded projective $B$-modules. 
We often regard $B^{\rm e}$ in a natural way as a $\ZZ$-graded algebra, and consider the category $\Gr B^{\rm e}$ of $\ZZ$-graded $B^{\rm e}$-modules.

For an abelian category $\Aa$, we denote by $\Cc(\Aa)$ the category of chain complexes, by $\Kk(\Aa)$ the homotopy category and by $\Dd(\Aa)$ the derived category.
We denote by $\Cc^{\rm b}(\Aa)$ the category of bounded chain complexes, by $\Kk^{\rm b}(\Aa)$ the bounded homotopy category and by $\Db(\Aa)$ the bounded derived category.

For a $k$-algebra $A$, we let $\Dd(A):=\Dd(\Mod A)$. We denote by $\per A$ the thick subcategory of $\Dd(A)$ generated by $A$. We denote by $\Dd^{\fd}(A)$ the full subcategory of $\Dd(A)$ consisting of objects $X$ satisfying $\dim_k(H^*(X))<\infty$. For a noetherian $k$-algebra $A$, we denote by $\Db(A)$ the full subcategory of $\Dd(A)$ consisting of objects $X$ satisfying $H^*(X)\in\mod A$.

We denote by $gf$ the composition of morphisms (or arrows) $f:X\to Y$ and $g:Y\to Z$.

\section{Background material}

In this section we give some background material on cluster tilting subcategories and on generalized cluster categories.

\subsection{Cohen-Macaulay modules over Iwanaga-Gorenstein algebras}

The following class of noetherian algebras was given by Iwanaga \cite{Iwa79}.

\begin{dfa}\label{def of IG}
A noetherian algebra $C$ is called \emph{Iwanaga-Gorenstein} if $\injdim_CC<\infty$ and $\injdim_{C^{\rm op}}C<\infty$.
\end{dfa}

For example, commutative local Gorenstein algebras and finite dimensional selfinjective algebras are clearly Iwanaga-Gorenstein.
Iwanaga-Gorenstein algebras have a distinguished class of modules defined as follows.

\begin{dfa}
Let $C$ be an Iwanaga-Gorenstein algebra.
The category $\CM(C)$ of \emph{(maximal) Cohen-Macaulay $C$-modules} is defined by
\[\CM(C):=\{X\in\mod C\ |\ \Ext^i_C(X,C)=0\ \mbox{ for any }i>0\}.\]
The \emph{stable category} $\underline{\CM}(C)$ has the same objects as $\CM(C)$,
and the morphisms spaces are given by
\[\Hom_{\underline{\CM}(C)}(X,Y):=\Hom_C(X,Y)/[C](X,Y)\]
where $[C](X,Y)$ consists of morphisms factoring through the smallest full subcategory $\add C$ of $\mod C$ stable under direct summands and containing $C$.
\end{dfa}

If $C$ is a local commutative Gorenstein algebra, then $\CM(C)$ is exactly the category of maximal Cohen-Macaulay $C$-modules.
If $C$ is a finite dimensional selfinjective algebra, then $\CM(C)$ is just $\mod C$.

Let us give basic properties of the category $\CM(C)$. 

\begin{prop}\label{CM over Iwanaga}
Let $C$ be an Iwanaga-Gorenstein algebra.
\begin{itemize}
\item[(a)] $\CM(C)$ is a Frobenius category and $\underline{\CM}(C)$ is a triangulated category \cite[Thm~2.6]{Hap88}.
\item[(b)] We have dualities $\xymatrix{\CM(C)\ar@<.2em>[rr]^{\Hom_C(-,C)}&&\CM(C^{\rm op})\ar@<.2em>[ll]^{\Hom_{C^{\rm op}}(-,C)}}$
which are mutually quasi-inverse and preserve the extension groups.
\item[(c)] We have a triangle equivalence $\underline{\CM}(C)\simeq \Db(C)/\per C$. \cite[Thm 4.4.1]{Buc87},\cite{KV87,Ric89}
\end{itemize}
\end{prop}

When an Iwanaga-Gorenstein algebra $C$ is a $\ZZ$-graded algebra, the category $\CM^{\ZZ}(C)$ of \emph{graded Cohen-Macaulay $C$-modules} is defined by
\[\CM^{\ZZ}(C):=\{X\in\gr C\ |\ \Ext^i_C(X,C)=0\ \mbox{ for any }i>0\}.\]
Then the stable category $\underline{\CM}^{\ZZ}(C)$ is defined similarly as above.

We have the following parallel results.

\begin{prop}
Let $C$ be a $\ZZ$-graded Iwanaga-Gorenstein algebra.
\begin{itemize}
\item[(a)] $\CM^{\ZZ}(C)$ is a Frobenius category and $\underline{\CM}^{\ZZ}(C)$ is a triangulated category.
\item[(b)] We have dualities $\xymatrix{\CM^{\ZZ}(C)\ar@<.2em>[rr]^{\Hom_C(-,C)}&&\CM^{\ZZ}(C^{\rm op})\ar@<.2em>[ll]^{\Hom_{C^{\rm op}}(-,C)}}$
which are mutually quasi-inverse and preserve the extension groups.
\item[(c)] We have a triangle equivalence $\underline{\CM}^{\ZZ}(C)\simeq \Db(\gr C)/\gr\per C$. 
\end{itemize}
\end{prop}

\subsection{$d$-Calabi-Yau categories and $d$-cluster tilting objects}

\begin{dfa}
A $k$-linear triangulated category $\Tt$ is said to be \emph{$d$-Calabi-Yau} if it is $\Hom$-finite and there is a functorial isomorphism
\[ \Hom_\Tt(X,Y)\simeq D\Hom_\Tt(Y,X[d]) \quad \textrm{for all } X,Y\in\Tt.\]
\end{dfa}

\begin{dfa}\cite{BMR+06},\cite[2.2]{Iya07a},\cite[2.1]{KR07}
A \emph{$d$-cluster tilting subcategory} $\Vv$ in a triangulated category $\Tt$ is a functorially finite subcategory of $\Tt$ such that
\begin{align*} \Vv & = \{ X\in \Tt, \Hom_\Tt(X,\Vv[i])=0, \ \forall\  1\leq i\leq d-1\} \\
& = \{ X\in \Tt, \Hom_\Tt(\Vv,X[i])=0, \ \forall\  1\leq i\leq d-1\}.
\end{align*}
An object $T\in\Tt$ is called \emph{$d$-cluster tilting} if the subcategory $\add(T)\subset \Tt$ is $d$-cluster tilting.
\end{dfa}

Cluster tilting subcategories are interesting because they determine
the triangulated category in the following sense:

\begin{prop}\label{criterion for equivalence}
Let $\Tt$ and $\Tt'$ be triangulated categories and $\Vv\subset\Tt$ and $\Vv'\subset \Tt'$ be $d$-cluster tilting subcategories.
If $\xymatrix{F:\Tt\ar[r] & \Tt'}$ is a triangle functor such that its restriction $F|_{\Vv}$ to $\Vv$ is an equivalence $\xymatrix{F|_{\Vv}:\Vv\ar[r] & \Vv'}$,
then $F$ is a triangle equivalence.
\end{prop}

\begin{proof}
The proposition is clear for $d=1$ since $\Tt=\Vv$ and $\Tt'=\Vv'$ hold in this case. 
It is proved in ~\cite[Lemma 4.5]{KR08} for $d\geq 2$. Note that the proof in~\cite{KR08} does not use the fact that $\Tt$ and $\Tt'$ are $d$-Calabi-Yau.
\end{proof}

\subsection{Generalized cluster categories}
Let $n\ge1$ be an integer.

Let $\Lambda$ be a finite dimensional algebra of global dimension at most $n$. 
Denote by
$\Theta=\Theta_n(\Lambda)$
a projective resolution of 
$$\RHom_{\Lambda}(D\Lambda,\Lambda)[n]\simeq\RHom_{\Lambda^{\rm e}}(\Lambda,\Lambda^{\rm e})[n]\simeq\RHom_{\Lambda^{\rm op}}(D\Lambda,\Lambda)[n] \quad \textrm{in }\Dd(\Lambda^{\rm e}).$$

\begin{dfa}\cite{Kel09,IO09}
We denote by $\Aa$ the differential graded category (DG category for short) of bounded complexes of finitely generated projective $\Lambda$-modules.
We define a DG functor by
$$F:=-\ten_\Lambda\Theta:\Aa\to\Aa.$$
The DG orbit category $\Aa/F$ has the same objects as $\Aa$, and
\begin{eqnarray*}
&\Hom_{\Aa/F}(X,Y):=\\
&{\rm colim}(\bigoplus_{\ell\ge0}\Hom_{\Aa}(F^\ell X,Y)\to
\bigoplus_{\ell\ge0}\Hom_{\Aa}(F^\ell X,FY)\to\bigoplus_{\ell\ge0}\Hom_{\Aa}(F^\ell X,F^2Y)\to\cdots).
\end{eqnarray*}
We denote by $\Dd(\Aa/F)$ the derived category of $\Aa/F$.
The \emph{generalized $n$-cluster category} $\Cc_n(\Lambda)$ is defined as the smallest thick subcategory
of $\Dd(\Aa/F)$ containing all representable functors of $\Aa/F$.
\end{dfa}

Let $\mathbb{S}=-\lten_\Lambda D\Lambda$ be the Serre functor of the category $\Db(\Lambda)$, and denote by $\SSS_n$ the composition $\SSS_n:=\SSS\circ[-n]$. Then we have an isomorphism $\SSS_n^{-1}\simeq -\ten_\Lambda \Theta$ of functors on $\Db(\Lambda)$.
From the construction of the generalized cluster category $\Cc_n(\Lambda)$, we have a triangle functor $\pi_\Lambda:\Db(\Lambda)\to\Cc_n(\Lambda)$ which induces
a fully faithful functor $\Db(\Lambda)/\SSS_n\to\Cc_n(\Lambda)$ for the orbit category $\Db(\Lambda)/\SSS_n$.

\begin{rema}\begin{itemize}
\item
For $n=2$ and an algebra $\Lambda$ of global dimension $1$, one
gets the usual cluster category $\Db(\Lambda)/\SSS_2$
constructed in~\cite{BMR+06}.
\item For $n=2$, and an algebra $\Lambda$ of global dimension 2, the construction is given in~\cite{Ami09} in the case where $\Cc_2(\Lambda)$ is $\Hom$-finite.
\item The generalization of results of~\cite{Ami09} from $2$ to $n\ge2$ is described in~\cite{Guo10}.
\end{itemize}
\end{rema}

The functor $\pi:\Db(\Lambda)\rightarrow \Cc_n(\Lambda)$ is
also described by a universal property (cf~\cite{Kel05, Ami09}). Here is
the version we will use in this paper (see appendix \cite{IO09}).

\begin{prop}\cite{Kel05, Ami09},\cite[Thm A.20]{IO09}\label{universal property}
Let $\Lambda$ be a finite dimensional algebra of global dimension
at most $n$. Let $C$ be an Iwanaga-Gorenstein algebra and $T$ be in $\Db(\Lambda^{\rm op}\ten C)$. If there
exists a morphism 
$T\to\Theta \ten_\Lambda T$ in $\Db(\Lambda^{\rm op}\ten C)$
whose cone is perfect as an object in
$\Db(C)$, then there exists a commutative diagram of triangle functors
\[\xymatrix{\Db(\Lambda)\ar[rr]^{-\lten_\Lambda T}\ar[d]^\pi && \Db(C)\ar[d]^{\rm nat.}\\
\Cc_n(\Lambda)\ar[rr] && \underline{\CM}(C).}\]
\end{prop}

Generalized cluster categories also have a nice description using certain DG algebras called \emph{derived preprojective algebras}.

\begin{dfa}\cite{Kel09,IO09}
Let $\Lambda$ be a finite dimensional algebra of global dimension at most $n$.
The \emph{derived $(n+1)$-preprojective algebra} of $\Lambda$ is defined as the tensor DG algebra
\[ \BPi_{n+1}(\Lambda):= \Talg_\Lambda(\Theta_n(\Lambda))=\Lambda\oplus \Theta\oplus (\Theta\ten_\Lambda \Theta)\oplus \ldots.\]
The \emph{$(n+1)$-preprojective algebra of $\Lambda$} is defined as the tensor algebra
\[ \Pi_{n+1}(\Lambda):=\Talg_\Lambda\Ext_\Lambda^n(D\Lambda,\Lambda)\simeq H^0(\BPi_{n+1}(\Lambda)).\]
\end{dfa}

The next result is shown in~\cite[Thm 4.10]{Ami09} for $n=2$.  The generalization to $n\ge2$ is done in~\cite{Guo10}.

\begin{thma}\cite{Ami09,Guo10},\cite[Thm 1.23]{Iya11}\label{cluster category}
Let $\Lambda$ be a finite dimensional algebra of global dimension
at most~$n$. Then the generalized $n$-cluster category
$\Cc_n(\Lambda)$ is $\Hom$-finite if and only if the $(n+1)$-preprojective algebra $\Pi_{n+1}(\Lambda)$ is finite dimensional. In this case, we have the following properties.
\begin{itemize}
\item[(a)] The category $\add\{\SSS_n^i\Lambda\mid i\in\ZZ\}$ is an $n$-cluster tilting subcategory of $\Db(\Lambda)$.
\item[(b)] The category
$\Cc_n(\Lambda)$  is $n$-Calabi-Yau, and the object
$\pi(\Lambda)$ is $n$-cluster tilting with endomorphism algebra
$\Pi_{n+1}(\Lambda)$.
\item[(c)] We have a triangle equivalence $\Cc_n(\Lambda)\simeq\per \BPi_{n+1}(\Lambda)/\Dd^{\fd}(\BPi_{n+1}(\Lambda))$.
\end{itemize}
\end{thma}

\section{Calabi-Yau algebras as higher Auslander algebras}

Under certain conditions on a bimodule $d$-Calabi-Yau algebra $B$ and an idempotent $e\in B$, we show in this section that $C:=eBe$ is an Iwanaga-Gorenstein algebra,
and that $Be$ is a $(d-1)$-cluster tilting object in the category $\CM(C)$ of Cohen-Macaulay $C$-modules.

\begin{dfa}\cite[(3.2.5)]{Gin06}\label{definition of CY algebra}
Fix an integer $d\geq 2$. We say that a $k$-algebra $B$ is \emph{bimodule $d$-Calabi-Yau} if
$B\in\per B^{\rm e}$ and $\RHom_{B^{\rm e}}(B,B^{\rm e})[d]\simeq B$ in $\Dd(B^{\rm e})$.
\end{dfa}

Note that if $B$ is bimodule $d$-Calabi-Yau, then so is $B^{\rm op}$.

\begin{exa}
Let $R=k[x_1,\cdots,x_d]$ be a polynomial algebra.
If an $R$-algebra $B$ is a finitely generated free $R$-module and satisfies $\Hom_R(B,R)\simeq B$ as $B^{\rm e}$-modules, then it is bimodule $d$-Calabi-Yau \cite[Thm 7.2.14]{Gin06},\cite[Thm 3.2]{IR08}.
\end{exa}

Let $B$ be a $k$-algebra, and $e$ an idempotent in $B$.
Assume that $B$ and $e(\neq1)$ satisfy the following conditions.

\smallskip
\hspace{.5cm}(A1)  $B$ is bimodule $d$-Calabi-Yau.

\smallskip
\hspace{.5cm}(A2) $B$ is noetherian.

\smallskip
\hspace{.5cm}(A3) $\underline{B}:=B/\langle e\rangle$ is a finite dimensional $k$-algebra.

The aim of this section is to prove the following results.

\begin{thma}\label{cluster tilting CM}
Let $B$ be a $k$-algebra, $e\in B$ be an idempotent and $C:=eBe$. Under assumptions (A1), (A2) and (A3), we have the following.
\begin{itemize}
\item[(a)] $C$ is an Iwanaga-Gorenstein algebra.
\item[(b)] $Be$ is a Cohen-Macaulay $C$-module.
\item[(c)] We have natural isomorphisms $\End_C(Be)\simeq B$ and $\End_{C^{\rm op}}(eB)\simeq B^{\rm op}$ which induce isomorphisms
$\End_{\underline{\CM}(C)}(Be)\simeq \B$ and $\End_{\underline{\CM}(C^{\rm op})}(eB)\simeq \B^{\rm op}$.
\item[(d)] $Be$ is $(d-1)$-cluster tilting in $\CM(C)$.
\end{itemize}
\end{thma}

The above statements (c) and (d) show that $B$ is a higher Auslander algebra of $C$ in the sense of \cite[Section 1]{Iya07b}.

If moreover $B$ is a graded $k$-algebra, we have the following additional information.

\begin{prop}\label{cluster tilting graded CM}
In addition to assumptions (A1), (A2) and (A3), assume that $B=\bigoplus_{\ell\ge0}B_\ell$ is a graded $k$-algebra such that $\dim_kB_\ell$ is finite for all $\ell\in\ZZ$. Then we have the following.
\begin{itemize}
\item[(a)] $Be$ is a graded Cohen-Macaulay $C$-module.
\item[(b)] The isomorphisms in Theorem \ref{cluster tilting CM} preserve the grading, i.e. they induce isomorphisms
\begin{eqnarray*}
&\Hom_{\Gr C}(Be, Be(\ell))\simeq B_\ell,\quad \Hom_{\Gr (C^{\rm op})}(eB, eB(\ell))\simeq B^{\rm op}_\ell,&\\
&\Hom_{\underline{\CM}^\ZZ (C)}(Be, Be(\ell))\simeq\underline{B}_\ell\quad \textrm{and}\quad \Hom_{\underline{\CM}^\ZZ (C^{\rm op})}(eB, eB(\ell))\simeq \underline{B}^{\rm op}_\ell.&
\end{eqnarray*} 
\item[(c)] The category $\add\{Be(i) \mid i\in \ZZ\}$ is a $(d-1)$-cluster tilting subcategory of $\CM^{\ZZ}(C)$.
\end{itemize}
\end{prop}

The proof of Theorem~\ref{cluster tilting CM} is given in the next two subsections. Assertions (a), (b) and (c) are proved in subsection~\ref{subsection Iwanaga Gorenstein}. Subsection~\ref{subsection cluster tilting} is devoted to the proof of (d).

\subsection{$C$ is Iwanaga-Gorenstein}\label{subsection Iwanaga Gorenstein}
In the rest of the section we assume that the algebra $B$ satisfies (A1), (A2) and (A3).

The following is a basic property of bimodule $d$-Calabi-Yau algebras.

\begin{prop}\label{CY give CY}
Let $B$ be a bimodule $d$-Calabi-Yau algebra.
\begin{itemize}
\item[(a)] \cite[Prop 3.2.4]{Gin06}\cite[Lemma 4.1]{Kel08} For any $X\in\Dd(B)$ and $Y\in\Dd^{\fd}(B)$, we have a functorial isomorphism
\[\Hom_{\Dd(B)}(X,Y)\simeq D\Hom_{\Dd(B)}(Y,X[d]).\]
In particular,  $\Dd^{\fd}(B)$ is a $d$-Calabi-Yau triangulated category.
\item[(b)] We have $\gldim B=d$.
\end{itemize}
\end{prop}

\begin{proof}
(b) For any $X,Y\in\Dd(B)$, it is easy to see that we have
\begin{eqnarray*}
\RHom_B(X,Y)\simeq\RHom_{B^{\rm e}}(B,\Hom_k(X,Y))\simeq\Hom_k(X,Y)\lten_{B^{\rm e}}\RHom_{B^{\rm e}}(B,B^{\rm e})\\
\simeq\Hom_k(X,Y)\lten_{B^{\rm e}}B[-d].
\end{eqnarray*}
In particular, for any $X,Y\in\Mod B$, we have
\[\Ext^{d+1}_B(X,Y)\simeq H^{d+1}(\Hom_k(X,Y)\lten_{B^{\rm e}}B[-d])=0.\]
Hence the global dimension of $B$ is at most $d$. It is exactly $d$ since
$\Ext^d_B(\B,B)\simeq D\Hom_B(B,\B)\neq 0$ holds by (A3) and (a).
\end{proof}

Let us make the following easy observations.

\begin{lema}\label{application of serre}
\begin{itemize}
\item[(a)] For any $X\in\fd B$, we have $\Ext_B^i(X,B)=0$ for any $i\neq d$.
\item[(b)] For any $X\in\mod\B$, we have $\Ext_B^i(X,eB)=0$ for any $i\in\ZZ$.
\end{itemize}
\end{lema}

\begin{proof}
We only prove (b) since (a) is simpler. Since $\dim_kX<\infty$ by (A3), we have
\[\Ext_B^i(X,eB)\simeq D\Ext_B^{d-i}(eB,X)\]
by Proposition~\ref{CY give CY}.
If $i\neq d$, then $\Ext_B^{d-i}(eB,X)$ is zero since $eB$ is projective.
If $i=d$, then it is zero since $X\in\mod\B$.
\end{proof}

\begin{prop}\label{CM property of Be}
We have 
$$\Ext_C^i(Be,C) \simeq  \left\{\begin{array}{cl} 0  &\textrm{if } i\neq0 \\  eB&  \textrm{if } i=0\end{array}\right.
\quad \textrm{and}\quad
\Ext_C^i(Be,Be) \simeq  \left\{\begin{array}{cl} 0  &\textrm{if } 1\le i\le d-2 \\  B&  \textrm{if } i=0.\end{array}\right.$$
\end{prop}

\begin{proof}
We consider
the triangle
\[\xymatrix{Be\lten_CeB\ar[r]^-{f} &B\ar[r]& X\ar[r]&Be\lten_CeB[1] \quad \textrm{ in } \Dd(B^{\rm e}),}\]
where $f$ is the composition $Be\lten_CeB\xrightarrow{}
Be\ten_CeB\xrightarrow{{\rm mult.}} B$ of natural maps.
Applying $-\lten_BBe$, we have an isomorphism $f\lten_BBe$.
Thus $X\lten_BBe=0$ holds. This means that
$H^i(X)e=0$ and hence $H^i(X)\in\mod\B$ for any $i\in\ZZ$. 

By Lemma \ref{application of serre}(b), we have $\RHom_B(X,eB)=0$. Applying
$\RHom_B(-,eB)$ to the above triangle, we get
\begin{eqnarray*}
&&eB=\RHom_B(B,eB)\simeq \RHom_B(Be\lten_CeB,eB)\\
&\simeq& \RHom_C(Be,\RHom_B(eB,eB))\simeq\RHom_C(Be,C)\quad \textrm{ in }\Dd(C^{\rm op}\ten B).
\end{eqnarray*} 
Thus the first assertion follows.

Similarly we have
\begin{eqnarray*}
\RHom_B(Be\lten_CeB,B)\simeq\RHom_C(Be,\RHom_B(eB,B))\simeq\RHom_C(Be,Be)\quad \textrm{ in }\Dd(C^{\rm op}\ten B).
\end{eqnarray*} 
Since $Be$ and $eB$ are concentrated in degree $0$,  $H^i(Be\lten_C eB)$ vanishes for $i>0$, and then $H^i(X)=0$ for any $i>0$. Hence we have $H^i(\RHom_B(X,B))=0$ for any $i<d$ again by Lemma \ref{application of serre}(a). 
Applying $\RHom_B(-,B)$ to the above triangle, we have an exact sequence
\[\xymatrix{\Hom_{\Dd(B)}(X,B[i])\ar[r]&\Hom_{\Dd(B)}(B,B[i])\ar[r]&
\Hom_{\Dd(B)}(Be\lten_CeB,B[i])\ar[r]&\Hom_{\Dd(B)}(X,B[i+1]).}\]
In particular, for any $i$ with $0\le i\le d-2$, we have isomorphisms
\[\Ext^i_C(Be,Be)\simeq \Hom_{\Dd(B)}(Be\lten_CeB,B[i])\simeq\Hom_{\Dd(B)}(B,B[i])\]
which show the second assertion.
\end{proof}

Now we are ready to prove Theorem~\ref{cluster tilting CM}(a), (b) and (c).

(i) First we show that $C$ is noetherian.

This follows from (A2) by the following easy argument:
Any right ideal $I$ of $C$ gives a right ideal $\widetilde{I}:=IB$ of $B$ satisfying $\widetilde{I}e=I$.
Thus any strictly ascending chain of right ideals of $C$ gives a strictly ascending chain of right ideals of $B$.
Thus $C$ is right noetherian. Similarly $C$ is left noetherian.

(ii) Next we show that $C$ is an Iwanaga-Gorenstein algebra.

For any $X\in\Mod C$, we shall show $\Ext^{d+1}_C(X,C)=0$.
Let $Y:=X\ten_CeB$ and $P_{\bullet}$ be a projective resolution of the $B$-module $Y$.
Then $P_{\bullet}e$ is a bounded complex in $\add_C(Be)$ 
which is quasi-isomorphic to $Ye\simeq X$.
Since by Proposition~\ref{CM property of Be} $\Ext^i_C(Be,C)$ vanishes for any $i>0$, we have
\[\Ext^{d+1}_C(X,C)\simeq H^{d+1}(\Hom_C(P_{\bullet}e,C)).\]
Since we have isomorphisms
\begin{eqnarray*}
\Hom_C(P_{\bullet}e,C)\simeq\Hom_C(P_{\bullet}\ten_BBe,C)
\simeq\Hom_B(P_{\bullet},\Hom_C(Be,C))\simeq\Hom_B(P_{\bullet},eB),
\end{eqnarray*}
we get
\[\Ext^{d+1}_C(X,C)\simeq H^{d+1}(\Hom_B(P_{\bullet},eB))\simeq\Ext^{d+1}_B(Y,eB)=0\]
by Proposition~\ref{CY give CY}.

(iii) We show that $Be$ is a Cohen-Macaulay $C$-module.

By Proposition~\ref{CM property of Be}, we only have to show that $Be$ is a finitely generated $C$-module.
By (A2), the right ideal $\langle e\rangle=BeB$ of $B$ is finitely generated.
There exists a finite generating set of the $B$-module $BeB$
which is contained in $Be$.
Clearly it gives a finite generating set of the $C$-module $Be$.

(iv) We show Theorem~\ref{cluster tilting CM}(c).

We have $\End_C(Be)\simeq B$ by Proposition~\ref{CM property of Be}.
Hence we have an equivalence
$$\Hom_C(Be,-):\add_C(Be)\to\proj B$$
which sends $C$ to $eB$. Thus we have
$$\End_{\underline{\CM}(C)}(Be)=\End_C(Be)/[C]\simeq \End_B(B)/[eB]\simeq B/BeB=\B.$$
Here we denote by $[C]$ (respectively, $[eB]$)the ideal of $\End_C(Be)$
(respectively, $\End_B(B)$) consisting of morphisms factoring through
$\add C$ (respectively, $\add eB$).

Similarly we have $B^{\rm op}\simeq\End_{C^{\rm op}}(eB)$ and $\B^{\rm op}\simeq \End_{\underline{\CM}(C^{\rm op})}(eB)$.
\qed

\medskip
We end this subsection with the following observation (which will not be used in this paper) asserting that $C$ enjoys the bimodule $d$-Calabi-Yau property except that $C$ may not be perfect as a bimodule over itself. 

\begin{rema}
We have $\RHom_{C^{\rm e}}(C,C^{\rm e})[d]\simeq C$ in $\Dd(C^{\rm e})$.
\end{rema}

\begin{proof}
Let $P_\bullet$ be  a projective resolution of the $B^{\rm e}$-module $B$.
Applying $eB\otimes_B-\otimes_BBe$, we get an isomorphism $eP_\bullet e\simeq C$ in $\Dd(C^{\rm e})$.
By Proposition~\ref{CM property of Be}, we have
\begin{eqnarray*}
\RHom_{C^{\rm e}}(eB\otimes Be,C^{\rm e})&=&
\RHom_{C^{\rm op}}(eB,C)\otimes \RHom_C(Be,C)\\
=\Hom_{C^{\rm op}}(eB,C)\otimes \Hom_C(Be,C)
&=&\Hom_{C^{\rm e}}(eB\otimes Be,C^{\rm e}).
\end{eqnarray*}
Thus each term $eP_ie$ in $eP_\bullet e$ satisfies $\Ext^i_{C^{\rm e}}(eP_ie,C^{\rm e})=0$ for any $i>0$, and we have
\[\RHom_{C^{\rm e}}(C,C^{\rm e})\simeq\Hom_{C^{\rm e}}(eP_\bullet e,C^{\rm e}).\]
Since the functor
\[eB\ten_B-\ten_BBe:\proj B^{\rm e}\to\mod C^{\rm e}\]
is fully faithful by Theorem \ref{cluster tilting CM}(c), we have
\[\Hom_{C^{\rm e}}(eP_ie,C^{\rm e})\simeq\Hom_{B^{\rm e}}(P_i,Be\otimes eB)=e\Hom_{B^{\rm e}}(P_i,B^{\rm e})e\]
Consequently we have
\begin{align*}
\RHom_{C^{\rm e}}(C,C^{\rm e})
&\simeq\Hom_{C^{\rm e}}(eP_\bullet e,C^{\rm e})\\
&\simeq e\Hom_{B^{\rm e}}(P_\bullet,B^{\rm e})e\\
&\simeq e\RHom_{B^{\rm e}}(B,B^{\rm e})e\\
&\simeq e(B[-d])e=C[-d].&&\qedhere
\end{align*}
\end{proof}

\subsection{$Be$ is $(d-1)$-cluster tilting}\label{subsection cluster tilting}
In this subsection we prove Theorem~\ref{cluster tilting CM}(d).

By Proposition~\ref{CM property of Be}, we have $\Ext^i_C(Be,Be)=0$ for any $i$ with $1\le i\le d-2$.
The assertion follows from the following lemmas.

\begin{lema}\label{second syzygy}
For any $X\in\mod C$, we have $\projdim_{B^{\rm op}}\Hom_C(X,Be)\le d-2$.
\end{lema}

\begin{proof}
Let $\xymatrix{P_1\ar[r]&P_0\ar[r]&X\ar[r]&0}$ be a projective presentation of $X$ in $\mod C$.
Applying $\Hom_C(-,Be)$, we have an exact sequence
\[\xymatrix{0\ar[r]&\Hom_C(X,Be)\ar[r]&\Hom_C(P_0,Be)\ar[r]&\Hom_C(P_1,Be)}\]
of $B^{\rm op}$-modules. 
Then $\Hom_C(P_i,Be)$ is a projective $B^{\rm op}$-module for $i=0,1$.
Since $\gldim B^{\rm op}=d$ by Proposition \ref{CY give CY},
we have $\projdim_{B^{\rm op}}\Hom_C(X,Be)\le d-2$
\end{proof}

\begin{lema}\label{cluster lemma 2}
If $X\in\CM(C)$ satisfies $\Ext^i_C(X,Be)=0$ for any $i$ with $1\le i\le d-2$, then we have $X\in\add_C(Be)$.
\end{lema}

\begin{proof}
Let
$$\xymatrix{0\ar[r]&\Omega^{d-2}X\ar[r]&P_{d-3}\ar[r]&\cdots\ar[r]&P_0\ar[r]&X\ar[r]&0}$$
be a projective resolution of the $C$-module $X$.
Applying $\Hom_C(-,Be)$, we get an exact sequence
$$\xymatrix@C=.5cm{0\ar[r]&\Hom_C(X,Be)\ar[r]&\Hom_C(P_0,Be)\ar[r]&\cdots\ar[r]&\Hom_C(P_{d-3},Be)\ar[r]&\Hom_C(\Omega^{d-2}X,Be)\ar[r] & 0}$$
of $B^{\rm op}$-modules, where we used that $\Ext^i_C(X,Be)=0$ for any $i$ with $1\le i\le d-2$. 
By Lemma~\ref{second syzygy}, we have $\projdim_{B^{\rm op}}\Hom_C(\Omega^{d-2}X,Be)\leq d-2$.
Since each $\Hom_C(P_i,Be)$ is a projective $B^{\rm op}$-module, it follows that $\Hom_C(X,Be)$ is a projective $B^{\rm op}$-module.
Thus we have $\Hom_C(X,C)=e\Hom_C(X,Be)\in\add_{C^{\rm op}}(eB)$ and
$$X\simeq\Hom_{C^{\rm op}}(\Hom_C(X,C),C)\in\add_C\Hom_{C^{\rm op}}(eB,C)=\add_C(Be)$$
by Propositions~\ref{CM over Iwanaga} and \ref{CM property of Be}.
\end{proof}

\begin{lema}\label{cluster lemma 3}
If $X\in\CM(C)$ satisfies $\Ext^i_C(Be,X)=0$ for any $1\le i\le d-2$, then we have $X\in\add_C(Be)$.
\end{lema}

\begin{proof}

Let $(-)^*:=\Hom_C(-,C):\CM(C)\to\CM(C^{\rm op})$ be the duality in Proposition~\ref{CM over Iwanaga}.
Then we have $(Be)^*=eB$ by Proposition~\ref{CM property of Be}.
Since the duality $(-)^*$ preserves the extension groups, we have $\Ext^i_{C^{\rm op}}(X^*,eB)=0$ for any $i$ with $1\le i\le d-2$.
Applying Lemma~\ref{cluster lemma 2} to $(B,C,Be,X):=(B^{\rm op},C^{\rm op},eB,X^*)$, we have $X^*\in\add_{C^{\rm op}}(eB)$.
Applying $(-)^*$ again, we have $X\in\add_{C}(Be)$.
\end{proof}

Now Theorem \ref{cluster tilting CM}(d) is a direct consequence of Lemmas~\ref{cluster lemma 2} and \ref{cluster lemma 3}.

\section{Graded Calabi-Yau algebras as higher preprojective algebras}

In this section, which is independent of Section 2, we work with a graded algebra $B=\bigoplus_{\ell\geq 0}B_\ell$ such that $\dim_k B_0$ is finite. We show under assumptions of $d$-Calabi-Yau type on $B$,
that $B$ is isomorphic to the $d$-preprojective algebra of $A:=B_0$. 
\subsection{Basic setup and main result}\label{subsection BP}

\begin{dfa}
Let $d\geq 2$. Assume that $B=\bigoplus_{\ell\geq 0}B_\ell$ is a positively $\ZZ$-graded $k$-algebra. We say that $B$
is \emph{bimodule $d$-Calabi-Yau of Gorenstein parameter~$1$} if $B\in\per B^{\rm e}$ and
there exists a graded projective resolution $P_\bullet$ of $B$ as a bimodule and an isomorphism 
\begin{equation}\label{selfduality}
P_\bullet\simeq P^{\vee}_\bullet[d](-1)\quad \textrm{in } \Cc^{\rm b}(\grproj B^{\rm e}),
\end{equation}
where we denote by $(-)^{\vee}=\Hom_{B^{\rm e}}(-,B^{\rm e}):\Cc^{\rm
b}(\grproj B^{\rm e})\to\Cc^{\rm b}(\grproj (B^{\rm e})^{\rm op})\simeq \Cc^{\rm b}(\grproj B^{\rm e})$ the natural duality induced by a canonical isomorphism $(B^{\rm e})^{\rm op}\simeq B^{\rm e}$.
\end{dfa}

\begin{rema} If for any $\ell\in\mathbb{N}$ the homogenous part $B_\ell$ is finite dimensional, then the category $\gr B$ is $\Hom$-finite and Krull-Schmidt. Hence the graded algebra $B$ is bimodule $d$-Calabi-Yau of Gorenstein parameter $1$ if and only if there exists an isomorphism $$\RHom_{B^{\rm e}}(B,B^{\rm e})[d](-1)\simeq B \quad\textrm{in } \Dd(\Gr B^{\rm e}).$$ In this case, the minimal projective resolution $P_\bullet$ of $B$ as a $B$-bimodule satisfies \eqref{selfduality}
\end{rema}

Throughout this section we assume
\medskip
\begin{equation*}\label{(A1*)}\tag{A1*} B \textrm { is bimodule d-Calabi-Yau of Gorenstein parameter 1}.\end{equation*}

\medskip 
The aim of this section is to prove the following.

\begin{thma}\label{thm_preproj_A}
Let $B$ be as above with $A:=B_0$ finite dimensional. Then we have the following.
\begin{itemize}
\item[(a)] $A$ is a finite dimensional $k$-algebra with $\gldim A\leq d-1$.
\item[(b)] The derived $d$-preprojective algebra $\BPi_d(A)$ is concentrated in degree zero.
\item[(c)] There exists an isomorphism $\Pi_d(A)\simeq B$ of $\ZZ$-graded algebras,
where $\Pi_d(A)$ is the $d$-preprojective algebra of $A$.
\end{itemize}
\end{thma}

Note that as a consequence of this Theorem, we obtain that $\dim_kB_\ell$ is finite for all $\ell\geq 0$ since $B_\ell\simeq \underbrace{\Ext^{d-1}_A(DA,A)\otimes_A \cdots \otimes_A \Ext^{d-1}_A(DA,A)}_{\ell \textrm{ times}}$.

The main step of the proof consists of the following intermediate result.

\begin{prop}\label{proptriangle}
Let $B$ be as above, $A:=B_0$ and $\Theta=\Theta_{d-1}(A)$ be a projective resolution of $\RHom_{A^{\rm e}}(A, A^{\rm e})[d-1]$ in $\Dd(A^{\rm e})$. Then there exists a triangle
$$\xymatrix{\Theta\ten_A B(-1)\ar[r]^-{\alpha} & B\ar[r]^{a} & A\ar[r] & \Theta\ten_A B(-1)[1]}\quad \textrm{in } \Dd(\Gr (A^{\rm op}\ten B))$$
where $a:B\to A$ is the natural surjection.
\end{prop}

Before proving Proposition~\ref{proptriangle} and Theorem~\ref{thm_preproj_A}, let us give an application.

\begin{dfa}\cite{HIO}
Let $n$ be a positive integer.
A finite dimensional algebra $A$ is called \emph{$n$-representation infinite} if $\gldim A\le n$ and $\SSS_n^{-i}A$ belongs to $\mod A$ for any $i\ge0$.
\end{dfa}

Clearly an algebra $A$ with $\gldim A\le n$ is $n$-representation infinite if and only if $\BPi_{n+1}(A)$ is concentrated in degree zero.
Thus we have the following immediate consequence. 

\begin{cora}\label{from CY to nrepinfin}
Let $B$ be a graded algebra which is bimodule $d$-Calabi-Yau of Gorenstein parameter $1$, with $\dim_k B_0<\infty$. Then $B_0$ is $(d-1)$-representation infinite.
\end{cora}

The $n$-representation infinite algebras are also called \emph{extremely quasi $n$-Fano} and studied from the viewpoint of non-commutative algebraic geometry in \cite{MM10}.
In particular, Corollary~\ref{from CY to nrepinfin} was proved in \cite[Thm 4.12]{MM10} using
quite different methods. We note that combining with Keller's result
\cite[Thm 4.8]{Kel09}, we have a bijection between bimodule $d$-Calabi-Yau
algebras of Gorenstein parameter $1$ and $(d-1)$-representation
infinite algebras (see \cite[Thm 4.35]{HIO}).

\subsection{Splitting the graded projective resolution}

Let us start with the following observation.

\begin{lema}\label{clear}
Let $B$ be a positively graded algebra, and $A=B_0$. 
Let $Q_\bullet$ be a complex in $\Cc^{\rm b}(\grproj B^{\rm e})$ such that each term is generated in degree zero.
\begin{itemize}
\item[(a)] The degree zero part $(Q_\bullet)_0$ is isomorphic to $A\otimes_BQ_\bullet\otimes_BA$ in $\Cc^{\rm b}(\proj A^{\rm e})$.
\item[(b)] We have isomorphisms $B\otimes_AA\otimes_BQ_\bullet\simeq Q_\bullet\simeq Q_\bullet\otimes_BA\otimes_AB$ in $\Cc^{\rm b}(\grproj B^{\rm e})$.
\end{itemize}
\end{lema}

Let $B$, $P_\bullet$, and $A=B_0$ be as in subsection \ref{subsection BP}. The following observation is crucial.

\begin{lema}\label{lemmasplitting}
In the setup above, the following assertions hold.
\begin{itemize}
\item[(a)] There exist complexes
\begin{eqnarray*}
&\xymatrix{Q_\bullet=(Q_{d-1}\ar[r] & \cdots\ar[r] & Q_1\ar[r] & Q_0)}&\textrm{ and }\\
&\xymatrix{R_\bullet=(R_{d-1}\ar[r] & \cdots\ar[r] & R_1\ar[r] & R_0)}&\textrm{ in }\Cc^{\rm b}(\grproj B^{\rm e})
\end{eqnarray*}
and a morphism $\xymatrix{f:R_\bullet(-1)\ar[r] & Q_\bullet}$ in $\Cc^{\rm b}(\grproj B^{\rm e})$ such
that $P_\bullet$ is the mapping cone of $f$ and each $Q_i$ and $R_i$ are generated in degree zero. 
\item[(b)] We have $R_\bullet\simeq Q^{\vee}_\bullet[d-1]$ and
$Q_\bullet\simeq R^{\vee}_\bullet[d-1]$ in $\Cc^{\rm b}(\grproj B^{\rm e})$.
\end{itemize}
\end{lema}

\begin{proof}
(a) Since the resolution $P_\bullet$ of $B$ is minimal, and since $B_i=0$ for any $i<0$,
each $P_i$ is generated in non-negative degrees. If $P_i$ has a generator in degree $a\ge0$, then
by the isomorphism~\eqref{selfduality} $P_{d-i}$ has a generator in degree $1-a$, which implies $1-a\ge0$.
Therefore $a$ has to be $0$ or $1$, and each $P_i$ is generated in degree $0$ or $1$.

For each $i=0,\ldots ,d$ we write $P_i:=P^0_i\oplus P^1_i(-1)$,
where all the indecomposable summands of $P^0_i$ and $P^1_i$ are generated in degree zero. 
By the isomorphism~\eqref{selfduality}, we have $P^1_i\simeq (P^0_{d-i})^\vee$ for any $i\in\ZZ$.
Since the $B^{\rm e}$-module $B$ is generated in degree zero, we have
$P^1_0=0$ and so $P^0_d=0$. Then  the map $d_i:P_i\rightarrow P_{i-1}$ can be written
$$\xymatrix{d_i:P_i^0\oplus P_i^1(-1)\ar[rr]^{\begin{bmatrix}a_i& b_i\\0&-c_i\end{bmatrix}}& & P^0_{i-1}\oplus P^1_{i-1}(-1)}$$
Therefore we have
$$ \xymatrix{P_\bullet= (P_d\ar[r] & P_{d-1}\ar[r] & \ldots \ar[r] & P_2\ar[r]^{d_2} & P_1\ar[r]^{d_1} & P_0) \\
Q_\bullet:= (0\ar[r] & P^0_{d-1}\ar[r]\ar[u] & \ldots\ar[r] & P^0_2\ar[r]^{a_2}\ar[u] & P^0_1\ar[r]^{a_1}\ar[u] & P^0_0)\ar[u] \\
R_\bullet(-1):= (0\ar[r] & P^1_{d}(-1)\ar[r]\ar[u] & \ldots\ar[r] &
P^1_3(-1)\ar[r]^{c_3}\ar[u]^{b_3} &
P^1_2(-1)\ar[r]^{c_2}\ar[u]^{b_2} & P^1_1(-1))\ar[u]^{b_1}}$$ Hence
$P_\bullet$ is the mapping cone of the morphism
$f:R_\bullet(-1)\rightarrow Q_\bullet.$ 

(b) We have an exact sequence
$$\xymatrix{0\ar[r]&Q_\bullet\ar[r]&P_\bullet\ar[r]&R_\bullet(-1)[1]\ar[r]&0}\quad \textrm{in } \Cc^{\rm b}(\grproj B^{\rm e}).$$
Applying $(-)^{\vee}(-1)[d]$ and using the isomorphism~\eqref{selfduality}, we have
an exact sequence
$$\xymatrix{0\ar[r]&R_\bullet^{\vee}[d-1]\ar[r]&P_\bullet\ar[r]&Q_\bullet^{\vee}(-1)[d]\ar[r]&0}\quad \textrm{in } \Cc^{\rm b}(\grproj B^{\rm e}).$$
Since $Q_\bullet$ is generated in degree zero and the degree zero part of $Q_\bullet^{\vee}(-1)[d]$ is zero, we have
$\Hom_{\Cc^{\rm b}(\grproj B^{\rm e})}(Q_\bullet,Q_\bullet^{\vee}(-1)[d])=0$.
Similarly $\Hom_{\Cc^{\rm b}(\grproj B^{\rm e})}(R_\bullet^{\vee}[d-1],R_\bullet(-1)[1])=0$ holds.
Thus we have a commutative diagram
$$\xymatrix{
0\ar[r]&Q_\bullet\ar[r]\ar@<.15em>[d]&P_\bullet\ar[r]\ar@{=}[d]&R_\bullet(-1)[1]\ar[r]\ar@<.15em>[d]&0\\
0\ar[r]&R_\bullet^{\vee}[d-1]\ar[r]\ar@<.15em>[u]&P_\bullet\ar[r]&Q_\bullet^{\vee}(-1)[d]\ar[r]\ar@<.15em>[u]&0}
$$
which implies $Q_\bullet\simeq R^{\vee}_\bullet[d-1]$ and $R_\bullet\simeq Q^{\vee}_\bullet[d-1]$.
\end{proof}

\begin{lema}\label{degree 0}
Let $Q_\bullet$ be as defined in Lemma \ref{lemmasplitting}. We have the following isomorphisms.
\begin{itemize}
\item[(a)] $A\otimes_BQ_\bullet\otimes_BA\simeq A$ in $\Dd(A^{\rm e})$.
\item[(b)] $A\otimes_BQ_\bullet\simeq B$ in $\Dd(\Gr A^{\rm op}\otimes B)$.
\end{itemize}
\end{lema}

\begin{proof}
(a) Since $P_\bullet$ is isomorphic to the mapping cone of $f:R_\bullet(-1)\rightarrow Q_\bullet$, we have an isomorphism
$$(P_\bullet)_0\simeq {\sf Cone}((R_\bullet)_{-1}\rightarrow(Q_\bullet)_0)\quad
\textrm{in }\Cc^{\rm b}(\proj A^{\rm e})$$
where $(X)_\ell$ is the degree $\ell$ part of the complex
$X\in \Cc^{\rm b}(\grproj B^{\rm e})$. Since $B$ is only in non-negative
degrees, then so is $R_\bullet$. Hence we have
$$(P_\bullet)_0\simeq (Q_\bullet)_0\quad \textrm{in }\Cc^{\rm b}(\proj A^{\rm e}).$$

Since $P_\bullet\simeq B$ in $\Dd(\Gr B^{\rm e})$, we have $(P_\bullet)_0\simeq B_0=A$ in $\Dd(A^{\rm e})$.
Therefore we get $A\ten_BQ_\bullet\ten_B A\simeq (Q_\bullet)_0\simeq A$ in $\Dd(A^{\rm e})$ by Lemma \ref{clear}.

(b) We have the following isomorphisms in $\Dd(\Gr (A^{\rm op}\ten B))$:
\begin{align*}
A\ten_B Q_\bullet&\simeq (A\ten_B Q_\bullet\ten_B A)\ten_AB&\textrm{by Lemma}~\ref{clear}\\
&\simeq A\lten_AB\simeq B&\textrm{by (a)}.&\qedhere
\end{align*}
\end{proof}

\begin{prop}\label{global dimension of A}
We have $\gldim A\le d-1$.
\end{prop}

\begin{proof}
By Lemma~\ref{degree 0} , $A\ten_B Q_\bullet\ten_B A$ is a projective resolution of the $A^{\rm e}$-module $A$. 
Thus we have $\gldim A\le\projdim_{A^{\rm e}}A\le d-1$. 
\end{proof}

\begin{lema}\label{degree 02}
Let $R_\bullet$ be as defined in Lemma \ref{lemmasplitting}. Then we have the following isomorphisms.
\begin{itemize}
\item[(a)] $A\otimes_BR_\bullet\otimes_BA\simeq\Theta$ in $\Dd(A^{\rm e})$.
\item[(b)] $A\otimes_BR_\bullet\simeq\Theta\ten_AB$ in $\Dd(\Gr A^{\rm op}\otimes B)$.
\end{itemize}
\end{lema}

\begin{proof}
(a) We have the following isomorphisms in $\Dd(A^{\rm e})$:
\begin{align*}
A\ten_B R_\bullet \ten_B A [1-d]& \simeq A\ten_B Q^{\vee}_\bullet\ten_B A& \textrm{by Lemma}~\ref{lemmasplitting}\\
&\simeq  A\ten_B \Hom_{B^{\rm e}}(Q_\bullet, B^{\rm e})\ten_B A \\
&\simeq  \Hom_{B^{\rm e}}(Q_\bullet, A^{\rm e}) \\
&\simeq  \Hom_{B^{\rm e}}(B\ten_AA\ten_BQ_\bullet\ten_BA\ten_AB, A^{\rm e})&\textrm{by Lemma}~\ref{clear}\\
&\simeq  \Hom_{A^{\rm e}}(A\ten_BQ_\bullet\ten_BA, A^{\rm e})\\
&\simeq  \RHom_{A^{\rm e}}(A, A^{\rm e})&\textrm{by Lemma}~\ref{degree 0}.
\end{align*}

(b) We get the following isomorphisms in $\Dd(\Gr (A^{\rm op}\ten B))$:
\begin{align*}
A\ten_B R_\bullet&\simeq (A\ten_B R_\bullet\ten_BA)\ten_A B&\textrm{by Lemma}~\ref{clear}\\
&\simeq \Theta\ten_AB&\textrm{by (a)}.&\qedhere
\end{align*}
\end{proof}

Now we are ready to prove Proposition~\ref{proptriangle}.

By Lemma~\ref{lemmasplitting} there exists a triangle
$\xymatrix{R_\bullet(-1)\ar[r] & Q_\bullet \ar[r] & P_\bullet \ar[r]
& R_\bullet(-1)[1]}$ in $\Dd(\Gr B^{\rm e}).$ Applying the
functor $A\lten_B-$ to this triangle we get the triangle
$$\xymatrix{A\ten_B R_\bullet( -1 )\ar[r] & A\ten_B Q_\bullet\ar[r]
& A\ten_B P_\bullet\ar[r] & A\ten_B R_\bullet(-1)[1]}\quad
\textrm{in }\Dd(\Gr (A^{\rm op}\ten B)).$$

By Lemmas~\ref{degree 0} and~\ref{degree 02}, we get a commutative diagram
$$\xymatrix@C=2em{A\ten_B R_\bullet(-1)\ar[r]\ar[d]^\wr & A\ten_B
Q_\bullet\ar[r]\ar[d]^\wr & A\ten_B P_\bullet\ar[r]\ar[d]^\wr & A\ten_B
R_\bullet[1](-1)\ar[d]^\wr\\
\Theta\ten_AB(-1)\ar[r] & B\ar[r]^a & A\ar[r] & \Theta\ten_A B(-1)[1]}$$
in $\Dd(\Gr (A^{\rm op}\ten B))$ with the natural surjection $a$.
\qed

\medskip
We end this subsection with recording the following observation, which is not used in this paper and follows easily from Lemmas ~\ref{degree 0} and~\ref{degree 02}.

\begin{rema}
We have isomorphisms $Q_\bullet\simeq B\lten_AB$ and $R_\bullet\simeq B\ten_A\Theta\ten_AB$ in $\Dd(A^{\rm e})$.
\end{rema}

\subsection{Proof of Theorem \ref{thm_preproj_A}}\label{subsection proof}

From Proposition~\ref{proptriangle}, we have a triangle
$$\xymatrix{\Theta\ten_A B(-1)\ar[r]^-{\alpha} & B\ar[r]^a & A\ar[r] & \Theta\ten_A B(-1)[1]}\quad \textrm{in }
\Dd(\Gr (A^{\rm op}\ten B)).$$ 
Since $a$ is the natural surjection, $\alpha$ is an isomorphism except for the degree zero part.

For any $\ell\geq 1$ we use the following notation:
\[\Theta^\ell:=\underbrace{\Theta\ten_A\Theta\ten_A\cdots\ten_A\Theta}_{\ell\textrm{ times}}\in\Dd(A^{\rm e}).\]

\begin{dfa}\label{alpha and beta}
Let $\alpha_\ell:\Theta^\ell\ten_A B\to B(\ell)$ be a morphism in $\Dd(\Gr(A^{\rm op}\ten B))$ defined as the composition 
$$\xymatrix{\alpha_\ell:\Theta^\ell\ten_A B\ar[rr]^-{1_{\Theta^{\ell-1}}\ten_A \alpha(1)} &&
\Theta^{\ell-1}\ten_A B(1)\ar[rr]^-{1_{\Theta^{\ell-2}}\ten_A \alpha(2)} && \cdots\ar[r] & \Theta\ten_A B(\ell-1)\ar[r]^-{\alpha(\ell)}  & B(\ell).}$$ 
For any $\ell\ge0$, the degree zero part of $\alpha_\ell$ is an isomorphism in $\Dd(A^{\rm e})$:
\[\xymatrix{(\alpha_\ell)_0:(\Theta^\ell\ten_AB)_0=\Theta^\ell\ar[r]^-\sim& B(\ell)_0=B_\ell}.\]
Applying $H^0$, we have an isomorphism in $\Mod (A^{\rm e})$:

$$\xymatrix{\beta_\ell:=H^0(\alpha_\ell)_0:H^0(\Theta^\ell)\ar[r]^-\sim & B_\ell.}$$
\end{dfa}

Now we are ready to prove Theorem~\ref{thm_preproj_A}.

(a) This is already shown in Proposition~\ref{global dimension of A}.

(b)
Since we have an isomorphism $(\alpha_\ell)_0:\Theta^\ell\to B_\ell$ in $\Dd(A^{\rm e})$ for any $\ell\ge0$, we have that $\BPi_d(A)=\Talg_\Lambda\Theta$ is concentrated in degree zero.

(c) Consider the following diagram for any $\ell,m\in\ZZ$:
$$\xymatrix@C=4em{
H^0(\Theta^\ell)\ten_AH^0(\Theta^m)\ar[r]_-\sim^-{1_{H^0(\Theta^\ell)}\ten_A\beta_m}\ar[d]_{\wr}&
H^0(\Theta^\ell)\ten_AB_m\ar[r]_-\sim^-{\beta_\ell\ten_A1_{B_m}}\ar[dr]_{H^0(\alpha_\ell)_m}&
B_\ell\ten_AB_m\ar[d]^{\rm mult.}\\
H^0(\Theta^{\ell+m})\ar[rr]^{\beta_{\ell+m}}_\sim && B_{\ell+m}
}$$
The left square commutes since $\alpha_{\ell+m}=\alpha_\ell(m)\circ(1_{\Theta^\ell}\ten_A\alpha_m)$ holds,
and the right triangle commutes since $H^0(\alpha_\ell):H^0(\Theta^\ell)\ten_AB\to B(\ell)$ is a morphism of right $B$-modules.
In particular, the $k$-linear isomorphism
$$\xymatrix{\bigoplus_{\ell\ge0}\beta_\ell:\Pi_d(A)=\bigoplus_{\ell\ge0}H^0(\Theta^\ell)\ar[r]^-\sim&B=\bigoplus_{\ell\ge0}B_\ell}$$
is compatible with the multiplication.
\qed

\medskip
The next lemma, which we will use later, follows immediately from the definitions of $\alpha_\ell$ and $\beta_\ell$.

\begin{lema}\label{commutative1}
\begin{itemize}
\item[(a)] The following diagram is commutative:
\[\xymatrix{H^0(\Theta^\ell)\ar[r]^-\sim\ar[dd]^{\beta_\ell}_{\wr} & \Hom_{\Dd(A)}(A,\Theta^\ell)\ar[d]^{-\lten_AB}\\
& \Hom_{\Dd(\Gr B)}(B, \Theta^\ell\ten_A B)\ar[d]^{\alpha_\ell\cdot}\\ B_\ell\ar[r]^-\sim &\Hom_{\Dd(\Gr B)}(B,B(\ell))
}\]
\item[(b)] $\beta_\ell$ is equal to the composition
$$\xymatrix{\beta_\ell:H^0(\Theta^\ell)\ar[r]_-\sim & H^0(\Theta)\ten_A\cdots\ten_AH^0(\Theta)\ar[rr]^-{\beta_1\ten_A\cdots\ten_A\beta_1}_-\sim && B_1\ten_A \ldots \ten_A B_1\ar[r]^-{\rm mult.} & B_\ell.}$$
\end{itemize}
\end{lema}

\section{Main results}

Let $B=\bigoplus_{\ell\geq 0}B_\ell$ be a positively $\ZZ$-graded algebra such that $\dim_k B_0<\infty$. Let $A:=B_0$ and let $e\in A$ be an idempotent. 
Assume that the conditions (A1*), (A2) and (A3) are satisfied, and in addition
\medskip

\hspace{.5cm}(A4) $eA(1-e)=0$.

\noindent
That is, we have an isomorphism of algebras $A\simeq \begin{bmatrix} eAe & 0\\ (1-e)Ae&\A\end{bmatrix}$. 
Combining Proposition~\ref{global dimension of A} and $(A4)$ we immediately get that $\gldim \A\leq d-1$.
Moreover recall from Section~2 that $C:=eBe$ is also noetherian and  that we have $Be\in\CM(C)$ and $eB\in\CM(C^{\rm op})$.

The aim of this section is to prove the following result.

\begin{thma}\label{main diagram}
Under assumptions (A1*), (A2), (A3) and (A4), we have the following.
\begin{itemize}
\item[(a)] The functor $\xymatrix{F:\Db(\A)\ar[r]^-{\rm Res.} & \Db(A)\ar[rr]^{-\lten_ABe} && \Db(\gr C)\ar[r] & \underline{\CM}^\ZZ(C)}$ is a triangle equivalence.
Moreover $Be$ is a tilting object in $\underline{\CM}^\ZZ(C)$.
\item[(b)] There exists a triangle equivalence $G:\Cc_{d-1}(\A)\to \underline{\CM}(C)$ making the diagram
$$\xymatrix{\Db(\A)\ar[rr]^{F}_\sim\ar[d]^{\pi} && \underline{\CM}^\ZZ(C)\ar^{\rm nat.}[d]\\ \Cc_{d-1}(\A) \ar[rr]^{G}_\sim && \underline{\CM}(C)}$$
commutative, where $\Cc_{d-1}(\A)$ is the generalized $(d-1)$-cluster category of $\A$.
\end{itemize}
\end{thma}
As a consequence we obtain that $\underline{\CM}(C)$ is $(d-1)$-Calabi-Yau.

\subsection{Notations and plan of the proof}

Let us start with some notations which we use in the proof. 

We denote as before by $\Theta=\Theta_{d-1}(A)$ a projective resolution of $\RHom_{A^{\rm e}}(A,A^{\rm e})[d-1]$ in $\Dd (A^{\rm e})$, and by $\underline{\Theta}=\Theta_{d-1}(\A)$ a projective resolution of $\RHom_{\A^{\rm e}}(\A,\A^{\rm e})[d-1]$ in $\Dd(\A^{\rm e})$.
For $\ell\geq 1$ we put
\[\Theta^\ell:=\underbrace{\Theta\ten_A\Theta\ten_A\cdots\ten_A\Theta}_{\ell\textrm{ times}}\in\Dd(A^{\rm e})\quad \textrm{and}\quad\underline{\Theta}^\ell:=\underbrace{\underline{\Theta}\ten_{\A}\underline{\Theta}\ten_{\A}\cdots\ten_{\A}\underline{\Theta}}_{\ell\textrm{ times}}\in\Dd(\underline{A}^{\rm e}).\]
We denote by $\Theta^{-1}$ a projective resolution of $DA[1-d]$ in $\Dd (A^{\rm e})$,
and by $\underline{\Theta}^{-1}$ a projective resolution of $D\A[1-d]$ in $\Dd (\A^{\rm e})$.
For $\ell\geq 1$ we put
\[\Theta^{-\ell}=\underbrace{\Theta^{-1}\ten_A\ldots\ten_A\Theta^{-1}}_{\ell \textrm{ times}}\in \Dd(\A^{\rm e})\ \mbox{ and }\ 
\underline{\Theta}^{-\ell}=\underbrace{\underline{\Theta}^{-1}\ten_{\A}\ldots\ten_{\A}\underline{\Theta}^{-1}}_{\ell \textrm{ times}}\in \Dd(\A^{\rm e}).\]
Then for any $\ell,m\in \ZZ$ we have isomorphisms $\Theta^{\ell}\ten_{A}\Theta^{m}\simeq \Theta^{\ell+m}$ in $\Dd(A^{\rm e})$ and
$\underline{\Theta}^{\ell}\ten_{\A}\underline{\Theta}^{m}\simeq \underline{\Theta}^{\ell+m}$ in $\Dd(\A^{\rm e})$.

\bigskip

The proof of Theorem~\ref{main diagram} is given in the next subsections. It consists of several steps which we outline here for the convenience of the reader. 

In subsection \ref{subsec1}, we construct for all $\ell\geq 0$  an isomorphism 
\begin{equation}\label{iso1} \Hom_{\Dd(\A)}(\A,\underline{\Theta}^\ell)\simeq \B_\ell \quad \textrm{(Lemma \ref{commutativity B E})}\end{equation} compatible with composition in $\Dd(\A)$ and product in $\B$.

\medskip

In subsection \ref{subsec2} we construct a map $\A\lten_A Be(1)\to \underline{\Theta}\lten_A Be$ in $\Dd(\Gr(\A^{\rm op}\ten C))$ whose cone is perfect as an object in $\Dd(\Gr C)$ (Proposition~\ref{prop cone perfect}).
With $F$ as in Theorem \ref{main diagram}(a), it gives us a commutative square for any $\ell\in\ZZ$
\begin{equation*}\label{commutative square}
\xymatrix{\Db(\A)\ar[rr]^{F} \ar[d]_{-\ten_{\A}\underline{\Theta}^\ell} && \underline{\CM}^\ZZ(C)\ar[d]^{(\ell)}\\ \Db(\A)\ar[rr]^{F} && \underline{\CM}^\ZZ(C) } \quad \textrm{(Proposition~\ref{compatibility degree theta})}\end{equation*}
and an isomorphism
\begin{equation}\label{iso2} F(\underline{\Theta}^\ell)\simeq Be(\ell) \quad \textrm{(Proposition~\ref{compatibility degree theta}).}\end{equation}
Moreover we can use this to show that $F$ induces a triangle functor $G:\Cc_{d-1}(\A)\to \underline{\CM}(C)$ (Proposition~\ref{existence of F}).
\medskip

In subsection \ref{subsec3} we show that the isomorphisms (\ref{iso1}) and (\ref{iso2}) are compatible with the map $F_{\A,\underline{\Theta}^\ell}$ for any $\ell\geq 0$, that is, there is a commutative diagram
\begin{equation*}\label{iso3}\xymatrix{\Hom_{\Db(\A)}(\A,\underline{\Theta}^\ell)\ar[rr]^-{F_{\A,\underline{\Theta}^\ell}} \ar[d]^\wr_{(\ref{iso1})}&& \Hom_{\underline{\CM}^\ZZ(C)}(F(\A),F(\underline{\Theta}^\ell))\ar[d]^\wr_{(\ref{iso2})} \\ \underline{B}_\ell \ar[rr]^-\sim_-{\rm Prop. \ref{cluster tilting graded CM}(b)}&& \Hom_{\underline{\CM}^\ZZ(C)}(Be,Be(\ell)) }.\end{equation*}
It implies that the map $F_{\A,\underline{\Theta}^\ell}$ is an isomorphism (Proposition \ref{big commutative diagram}).

\medskip

The last step of the proof consists of using $(d-1)$-cluster tilting subcategories in the categories $\Db(\A)$ and $\underline{\CM}^\ZZ(C)$, (resp. $\Cc_{d-1}(\A)$ and $\underline{\CM}(C)$) and Proposition~\ref{criterion for equivalence} to show that $F:\Dd(\A)\to \underline{\CM}^\ZZ(C)$ (resp. $G:\Cc_{d-1}(\A)\to \underline{\CM}(C)$) is a triangle equivalence.

\subsection{Preprojective algebras}\label{subsec1}

Using the following observation, we identify $\A\ten_A\Theta\ten_A\A$ and $\underline{\Theta}$ in the rest of this section.

\begin{lema}\label{lem3}
We have an isomorphism $\A\ten_A\Theta\ten_A\A\xrightarrow{}\underline{\Theta}$ in $\Dd(\A^{\rm e})$.
\end{lema}

\begin{proof}
We have the following isomorphism $$\A\ten_A \Theta\simeq \RHom_{A}(DA,\A)[d-1]\quad \textrm{in }\Dd(\A^{\rm op}\ten A).$$
Let $I_\bullet$ be an injective resolution of $\A$ as an
$\A^{\rm e}$-module. It follows from (A4) that $I_\bullet$ is also an injective
resolution of $\A$ as an $A$-module.
Hence we have the following isomorphisms in $\Dd(\A^{\rm e})$:
\begin{align*}
\A\ten_A\Theta\ten_A\A[1-d]& \simeq \RHom_A(DA,\A)\lten_A \A\\
 &\simeq  \Hom_A(DA,I_\bullet)\ten_A\A \\
 & \simeq  \Hom_{A^{\rm op}}(D I_\bullet,A)\ten_A \A \\
 & \simeq  \Hom_{ A^{\rm op}}(DI_\bullet, \A)\\
 &\simeq  \Hom_{\A^{\rm op}}(DI_\bullet, \A)\\
 & \simeq  \Hom_{\A}(D\A,I_\bullet) \simeq\underline{\Theta}[1-d].&&\qedhere
\end{align*}
\end{proof}

Denote by $p_0:A\to\A$ the natural projection in $\Mod (A^{\rm e})$.  For $\ell\geq 1$ we define the map $p_\ell:\Theta^\ell\to \underline{\Theta}^\ell$ in $\Dd(A^{\rm e})$ as the following composition: 
$$\xymatrix{\Theta^\ell\simeq A\ten_A \Theta\ten_A A\ten_A \Theta\ten_A\cdots\ten_A \Theta\ten_A A\ar[d]^-{p_0\ten_A 1_\Theta\ten_A p_0\ten_A\cdots \ten_A p_0}  \\  \A\ten_A \Theta\ten_A \A\ten_A \Theta\ten_A\cdots\ten_A \Theta\ten_A \A\ar[d]^{\wr}   \\ 
(\A\ten_A \Theta\ten_A \A)\ten_{\A}(\A\ten_A  \Theta\ten_A\cdots \ten_{\A}(\A\ten_A \Theta\ten_A \A)\simeq \underline{\Theta}^\ell.} $$  

\begin{lema}\label{commutativity B E}
Let $\beta_\ell:H^0(\Theta^\ell)\xrightarrow{\sim} B_\ell$ be as in Definition \ref{alpha and beta}.
Then there exists an isomorphism $H^0(\underline{\Theta}^\ell)\xrightarrow{\sim} \B_\ell$ making the following diagram commutative.
$$\xymatrix{H^0(\Theta^\ell)\ar[rr]_-\sim^-{\beta_\ell}\ar[d]^{H^0(p_\ell)} && B_\ell\ar[d]^{\rm nat.} \\ H^0(\underline{\Theta}^\ell)\ar[rr]_-\sim && \B_\ell.}$$
\end{lema}

\begin{proof}
Let $E:=H^0(\Theta)$, $\underline{E}:=H^0(\underline{\Theta})$ and for $\ell\geq 1$ 
$$E^\ell:=\underbrace{E\ten_A E\ten_A\ldots \ten_A E}_{\ell \textrm{ times}}\quad \textrm{and} \quad \underline{E}^\ell:=\underbrace{\underline{E}\ten_{\A} \underline{E}\ten_{\A}\ldots \ten_{\A} \underline{E}.}_{\ell \textrm{ times}}$$
Then we have isomorphisms $E^\ell\simeq H^0(\Theta^\ell)$ and $\underline{E}^\ell\simeq H^0(\underline{\Theta}^\ell)$.

(i) We show that $\beta_1:E\xrightarrow{\sim} B_1$ induces an isomorphism $\underline{E}\xrightarrow{\sim}\B_1$.

Taking $H^0$ of the isomorphism $\underline{\Theta}\simeq\A\ten_A\Theta\ten_A \A$ constructed in Lemma~\ref{lem3}, we obtain isomorphisms 
$$\underline{E}\simeq \underline{A}\ten_A E\ten_A \A\simeq \dfrac{E}{AeE+EeA}\simeq\dfrac{B_1}{AeB_1+B_1eA}\simeq\underline{B}_1\quad \textrm{in } \Mod(\A^{\rm e}).$$

(ii) We show that $\underline{E}\xrightarrow{\sim}\B_1$ in (i) induces an isomorphism $\underline{E}^\ell\xrightarrow{\sim}\B_\ell$ for any $\ell\ge1$.

Note that for $M$ and $N$ in $\Mod(\A^{\rm e})$ we have a canonical isomorphism $M\ten_A N\simeq M\ten_{\A}N$. Thus we have the following isomorphisms
$$\underline{E}^\ell\simeq \dfrac{E}{AeE+EeA}\ten_A\ldots\ten_A \dfrac{E}{AeE+EeA}\simeq \dfrac{E^\ell}{\sum_{i=0}^\ell E^ieE^{\ell-i}}\simeq \left( \dfrac{T_A E}{(e)}\right)_\ell.$$
Using the isomorphism of $\ZZ$-graded algebras $T_AE\simeq B$ in Theorem~\ref{thm_preproj_A}, we obtain $$\underline{E}^\ell\simeq \left( \dfrac{T_A E}{(e)}\right)_\ell\simeq \left(\dfrac{B}{(e)}\right)_{\ell}\simeq \B_\ell.$$

(iii) We show that the natural map
$$\xymatrix{{\rm nat.}:E^\ell\ar[rrr]^-{H^0(p_1)\ten_A\ldots \ten_A H^0(p_1)} &&&\underline{E}\ten_A\ldots \ldots \ten_A\underline{E}\simeq\underline{E}\ten_{\A}\ldots \ldots \ten_{\A}\underline{E} =\underline{E}^\ell}$$ 
makes the following diagram commutative:
$$\xymatrix{
H^0(\Theta^\ell)\ar[d]^{H^0(p_\ell)}\ar[rr]_-\sim&&E^\ell\ar[d]^{\rm nat.}\ar[rr]^-{\beta_1\ten_A\cdots\ten_A\beta_1}_-\sim&& B_1\ten_{A} \ldots \ten_{A} B_1\ar[rr]^-{\rm mult.} &&B_\ell\ar[d]^{\rm nat.}\\
H^0(\underline{\Theta}^\ell)\ar[rr]_-\sim&&\underline{E}^\ell\ar[rrrr]^-{\rm (ii)}_-\sim &&&& \underline{B}_\ell.}$$

The right pentagon is clearly commutative since both horizontal maps are induced by the isomorphism of $\ZZ$-graded algebras $T_AE\simeq B$.

We then show that the left square is commutative.
Since the square
$$\xymatrix{A\ten_A A\ar[rr]^{p_0\ten_A p_0}\ar[d]^{\wr} && \A\ten_A \A\ar[r]^\sim & \A\ten_{\A}\A\ar[d]^{\wr}\\ A\ar[rrr]^{p_0}&&& \A}$$
is clearly commutative, we have the assertion from the following isomorphisms:
$$\begin{array}{rcl}(H^0(p_1))^{\ten_A\ell}& \simeq &(H^0(p_0\ten_{A}1_{\Theta}\ten_A p_0))^{\ten_A\ell}\\& \simeq & H^0(p_0)\ten_A(1_{H^0(\Theta)}\ten_A H^0(p_0\ten_Ap_0))^{\ten_A\ell -1} \ten_A1_{H^0(\Theta)} \ten_AH^0(p_0)\\& \simeq & H^0(p_0)\ten_A(1_{H^0(\Theta)}\ten_A H^0(p_0))^{\ten_A\ell -1} \ten_A1_{H^0(\Theta)} \ten_AH^0(p_0)\\ 
&\simeq & H^0(p_\ell).\end{array}$$

(iv) Now the assertion follows from the commutative diagram in (iii) since the upper horizontal map is $\beta_\ell$ by Lemma~\ref{commutative1}.
\end{proof}

From Lemma \ref{commutativity B E}, we immediately get the following consequence.
\begin{cora}
We have an isomorphism $\Pi_d(\A)\simeq\B$ of $\ZZ$-graded algebras.
\end{cora}

By hypothesis $(A3)$, the algebra $\B$ is finite dimensional. Therefore we get the following consequence of Theorem~\ref{cluster category}.

\begin{cora}\label{cor existence cluster tilting}
Let $\Cc_{d-1}(\A)$ be the generalized $(d-1)$-cluster category associated to $\A$. Then the following hold.
\begin{itemize}
\item[(a)] $\Cc_{d-1}(\A)$ is a $(d-1)$-Calabi-Yau triangulated category.
\item[(b)] The object $\pi(\A)$ is a $(d-1)$-cluster tilting object in $\Cc_{d-1}(\A)$.
\item[(c)] The category $\add\{\underline{\Theta}^\ell \mid \ell\in \ZZ\}\subset \Db(\A)$ is a $(d-1)$-cluster tilting subcategory of $\Db(\A)$.
\end{itemize}
\end{cora}

\subsection{Compatibility of gradings}\label{subsec2}

Using the isomorphism $\A\ten_A \Theta\ten_A \A\simeq \underline{\Theta}$ in Lemma~\ref{lem3}, we prove the following.

\begin{lema}\label{lemM}
For any $M\in \mathcal{D}^{\rm b}(\A)$, the cone of the map 
$$\xymatrix{M\ten_A\Theta\lten_A Be\ar[rrr]^-{1_{M\ten_A\Theta}\ten_A p_0\lten_A 1_{Be}} &&& M\ten_A \Theta\ten_A \A\lten_A Be\simeq  M\ten_{\A}\underline{\Theta}\lten_A Be}$$ is perfect as an object in $\Dd(\Gr C)$.
\end{lema}

\begin{proof} From the triangle $\xymatrix{AeA\ar[r] & A\ar[r]^{p_0} & \A\ar[r] & AeA[1]}$ in $\mathcal{D}(A^{\rm e})$ we deduce that the cone of $(1_{M\ten_A\Theta})\ten_Ap_0\lten_A 1_{Be}$ is $(M\ten_A\Theta\ten_A AeA)\lten_A Be$. Since $A$ has finite global dimension, the object $M\ten_A\Theta$ is in $\per A$. 
So the object $M\ten_A\Theta\ten_A AeA$ is in $\thick(AeA)$, which 
is contained in $\thick(eA)$ by hypothesis (A4). Thus $(M\ten_A\Theta\ten_A AeA)\lten_A Be\in \thick (eBe)=\per C$. 
\end{proof}

For $\ell\geq 1$ we consider the map 
\[\gamma_\ell:=\alpha_\ell\lten_{B}1_{Be}:\Theta^\ell\ten_ABe\to Be(\ell)\quad \textrm{ in }\Dd(\Gr(A^{\rm op}\ten C)).\]

\begin{lema}\label{lem1}
The morphism $1_{\A}\lten_A\gamma_1:\A\ten_A\Theta\ten_ABe\to\A\lten_A Be(1)$ is an isomorphism in $\Dd(\Gr (\A^{\rm op}\ten C))$.
\end{lema}

\begin{proof}
The cone of this morphism is $\A\ten_A A(1)\ten_B Be=\A\ten_B Be(1)=\A e(1)=0$, so we have the assertion.
\end{proof}

From Lemmas~\ref{lemM} and~\ref{lem1} we get the following fundamental consequences.

\begin{prop}\label{prop cone perfect}
The cone of the composition map
$$\xymatrix@C=1.5cm{\A\lten_A Be(1)\ar[r]^-{(1_{\A}\lten_A \gamma_1)^{-1}}& \A\ten_A \Theta\ten_A Be\ar[rr]^-{(1_{\A\ten_A\Theta})\ten_Ap_0\lten_A 1_{Be}} && \A\ten_A\Theta\ten_A\A\lten_ABe\simeq
\underline{\Theta}\lten_A Be}$$
in $\Dd(\Gr(\A^{\rm op}\ten C))$ is perfect as an object in $\Dd(\Gr C)$.
\end{prop}

\begin{prop}\label{compatibility degree theta}
The functor $\xymatrix{F:\Db(\A)\ar[r]^-{\rm Res.} & \Db(A)\ar[rr]^{-\lten_ABe} && \Db(\gr C)\ar[r] & \underline{\CM}^\ZZ(C)}$ make the following diagrams commute up to isomorphism:
$$\xymatrix{\Db(\A)\ar[rr]^{F} \ar[d]_{-\ten_{\A}\underline{\Theta}} && \underline{\CM}^\ZZ(C)\ar[d]^{(1)}&&\Db(\A)\ar[rr]^{F} \ar[d]_{-\ten_{\A}\underline{\Theta}^{-1}} && \underline{\CM}^\ZZ(C)\ar[d]^{(-1)}\\ \Db(\A)\ar[rr]^{F} && \underline{\CM}^\ZZ(C) &&\Db(\A)\ar[rr]^{F} && \underline{\CM}^\ZZ(C).}$$
In particular, for any $\ell\in \ZZ$ we have $F(\underline{\Theta}^\ell)\simeq Be(\ell)$ in $\underline{\CM}^\ZZ(C)$.
\end{prop}

\begin{proof}
Since Proposition \ref{prop cone perfect} implies
\[(1)\circ F=(-\otimes_{\A}(\A\lten_A Be(1)))\simeq(-\otimes_{\A}(\underline{\Theta}\lten_A Be))=F\circ(-\ten_{\A}\underline{\Theta}),\]
we have the left diagram. The right diagram is an immediate consequence.
\end{proof}

Combining Proposition \ref{prop cone perfect}  with the universal property of the generalized cluster category (Proposition \ref{universal property}), we get the following consequence.

\begin{prop}\label{existence of F}
There  exists a triangle functor $G:\Cc_{d-1}(\A)\to \underline{\CM}(C)$ such that we have a commutative diagram
$$\xymatrix{\Db(\A)\ar[rr]^{F} \ar[d]_{\pi}&& \underline{\CM}^\ZZ(C)\ar[d]^{\rm nat.}\\ \Cc_{d-1}(\A)\ar[rr]^{G} && \underline{\CM}(C).}$$
\end{prop}

\begin{proof}
Let $T:=\A\lten_A Be$. Then Proposition \ref{prop cone perfect} gives a map $T\to\underline{\Theta}\lten_{\A}T$ in $\Dd(\A^{\rm op}\ten C)$ whose cone is perfect as an object in $\Dd(C)$.
Thus the assertion follows from Proposition \ref{universal property}.
\end{proof}

For any $\ell\geq 0$, we consider the map
$$q_\ell:=p_\ell\lten_A 1_{Be}:\Theta^\ell\ten_A Be\to \underline{\Theta}^\ell\lten_A Be\quad \textrm{ in } \underline{\CM}^\ZZ(C).$$
This is an isomorphism for $\ell=0$ since we have $AeA\in\thick(eA)$ and $eA\lten_ABe=C$.

The following isomorphism in $\underline{\CM}^{\ZZ}(C)$ plays an important role.

\begin{prop}\label{first step of induction}
The morphism in Proposition~\ref{prop cone perfect} gives an isomorphism
\[\xymatrix{\delta:F(\underline{\Theta})=\underline{\Theta}\lten_ABe\ar[r]^-\sim& \A\lten_ABe(1)=F(\A)(1)}\quad \textrm{ in }\underline{\CM}^{\ZZ}(C)\]
such that the following diagram commutes:
\begin{equation*}
\xymatrix{  
\Theta\ten_A Be\ar[rr]^{q_1}\ar[d]^{\gamma_1} && \underline{\Theta}\lten_ABe\ar[d]^-\delta\\
A\ten_A Be(1)\ar[rr]^{q_0(1)} && \A\lten_ABe(1)}
\end{equation*}
\end{prop}

\begin{proof}
Consider the following diagram: 
\[\xymatrix{
\Theta\ten_ABe\ar^-{p_0\ten_A(1_{\Theta\ten_ABe})}[rd]\ar^{\gamma_1}[dd]\ar[rrrrr]^{q_1}&&&&&
\underline{\Theta}\lten_ABe\ar^{\delta}[dd]\\
&\A\ten_A\Theta\ten_ABe\ar^-{(1_{\A\ten_A\Theta})\ten_Ap_0\lten_A 1_{Be}}[rrr]\ar^(.7){1_{\A}\lten_A\gamma_1}[drrrr]&&&\A\ten_A\Theta\ten_A\A\lten_ABe\ar[ru]_-\sim&\\
A\ten_ABe(1)\ar^{q_0(1)=p_0\lten_A1_{Be(1)}}[rrrrr]&&&&&\A\lten_ABe(1)
}\]
The upper square is commutative by definition of $q_1$, and the right square is commutative by definition of $\delta$.
The left square is commutative since both compositions are $p_0\lten\gamma_1$. Thus the assertion follows.
\end{proof}

For any $\ell\ge 1$, let $\delta_\ell:\underline{\Theta}^\ell\lten_A Be\to\A\lten_ABe(\ell)$ be an isomorphism in $\underline{\CM}^{\ZZ}(C)$ defined as the composition 
$$\xymatrix{\delta_\ell:\underline{\Theta}^\ell\lten_A Be\ar[rr]^-{1_{\underline{\Theta}^{\ell-1}}\lten_A \delta} &&
\underline{\Theta}^{\ell-1}\lten_A Be(1)\ar[rr]^-{1_{\underline{\Theta}^{\ell-2}}\lten_A \delta(1)} && \cdots\ar[r] & \underline{\Theta}\lten_A Be(\ell-1)\ar[r]^-{\delta(\ell-1)}  & Be(\ell).}$$
Then $\delta_\ell$ gives the isomorphism $F(\underline{\Theta}^\ell)=\underline{\Theta}^\ell\lten_A Be\to Be(\ell)$ in $\underline{\CM}^\ZZ(C)$ given in Proposition~\ref{compatibility degree theta}.

\subsection{$F$ and $G$ are triangle equivalences}\label{subsec3} 

The following result is the key step for proving that the triangle functors $F$ and $G$ are triangle equivalences. 

\begin{prop}\label{big commutative diagram}
The map $$F_{\underline{\Theta}^m,\underline{\Theta}^\ell}:\Hom_{\Dd(\A)}(\underline{\Theta}^m, \underline{\Theta}^\ell)\to \Hom_{\underline{\CM}^\ZZ(C)}(\underline{\Theta}^m\lten_A Be,\underline{\Theta}^\ell\lten_A Be)$$ is an isomorphism for any $m,\ell\in \ZZ$. 
\end{prop}

In order to prove this we need the following intermediate lemmas.

\begin{lema}\label{1st commutativity}
The isomorphism $\B_\ell\simeq\Hom_{\underline{\CM}^{\ZZ}(C)}(Be,Be(\ell))$ of Proposition~\ref{cluster tilting graded CM}(b) makes the following diagram commutative:
\[\xymatrix{
H^0(\Theta^\ell)\ar^-{\sim}[r]\ar[d]_\wr^{\beta_\ell}&\Hom_{\Dd(A)}(A,\Theta^\ell)\ar^{-\lten_ABe}[d]\\ B_\ell\ar^{{\rm nat.}}[d]
&\Hom_{\underline{\CM}^{\ZZ}(C)}(Be,\Theta^\ell\ten_ABe)\ar^{\gamma_\ell\cdot}[d]\\
\B_\ell\ar^-{\sim}[r]&\Hom_{\underline{\CM}^{\ZZ}(C)}(Be,Be(\ell))
}\]
\end{lema}

\begin{proof}
The above diagram is a part of the following:
\[\xymatrix{H^0(\Theta^\ell)\ar[r]^-\sim\ar[dd]^{\beta_\ell} & \Hom_{\Dd(A)}(A,\Theta^\ell)\ar[d]^{-\lten_A B} & \\ 
& \Hom_{\Dd(\Gr B)}(B,\Theta^\ell\ten_A B)\ar[r]^-{-\ten_BBe}\ar[d]^{\alpha_\ell\cdot} & \Hom_{\Dd(\Gr C)}(Be,\Theta^\ell\ten_A Be)\ar[d]^{\gamma_\ell\cdot} \\
B_\ell\ar[r]^-\sim \ar[d]^{\rm nat.}& \Hom_{\Dd (\Gr B)}(B,B(\ell))\ar[r]^-{-\ten_BBe} & \Hom_{\Dd(\Gr C)}(Be, Be(\ell))\ar[d]\\
\B_\ell\ar[rr]^-\sim && \Hom_{\underline{\CM}^\ZZ(C)}(Be,Be(\ell))}\]
The upper left pentagon is commutative by Lemma~\ref{commutative1}. The upper right square is commutative since by definition $\gamma_\ell=\alpha_\ell\lten_{B}1_{Be}$. The lower pentagon is commutative since the isomorphism of $\ZZ$-graded algebras
$\B\simeq \bigoplus_{\ell\in \ZZ}\Hom_{\underline{\CM}^\ZZ(C)}(Be,Be(\ell))$ is induced by the isomorphism of $\ZZ$-graded algebras $B\simeq \bigoplus_{\ell\in \ZZ}\Hom_{\Gr B}(B,B(\ell))$ (Proposition~\ref{cluster tilting graded CM}(b)).
Hence the original diagram is commutative.
\end{proof}

\begin{lema}\label{2nd commutativity}
For any $\ell\geq 0$ the following diagram commutes.
\[\xymatrix@C=3em{
H^0(\Theta^\ell)\ar_{\wr}[d]\ar^{H^0(p_\ell)}[r]&H^0(\underline{\Theta}^\ell)\ar_{\wr}[d]  \\
\Hom_{\Dd(A)}(A,\Theta^\ell)\ar^-{-\lten_ABe}[d]&\Hom_{\Dd(\A)}(\A,\underline{\Theta}^\ell) \ar^-{-\lten_ABe=F_{\A,\underline{\Theta}^\ell}}[d]\\
\Hom_{\underline{\CM}^{\ZZ}(C)}(Be,\Theta^\ell\ten_ABe)\ar^-{q_\ell\ \cdot\ q_0^{-1}}[r]&\Hom_{\underline{\CM}^{\ZZ}(C)}(\A\lten_ABe,\underline{\Theta}^\ell\lten_ABe)
}\]

\end{lema}

\begin{proof}

The above diagram is a part of the following, where ${}_C(-,-)$ is $\Hom_{\underline{\CM}^{\ZZ}(C)}(-,-)$:
\[\xymatrix@C=3em{
H^0(\Theta^\ell)\ar_{\wr}[d]\ar^{H^0(p_\ell)}[r]&H^0(\underline{\Theta}^\ell)\ar_{\wr}[d]\ar[r]^\sim&\Hom_{\Dd(\A)}(\A,\underline{\Theta}^\ell)\ar[d]^{{\rm nat.}}\\ 
\Hom_{\Dd(A)}(A,\Theta^\ell)\ar^-{-\lten_ABe}[d]\ar[r]^{p_\ell\cdot}& \Hom_{\Dd(A)}(A,\underline{\Theta}^\ell)\ar^-{-\lten_ABe}[d]& \Hom_{\Dd(A)}(\A,\underline{\Theta}^\ell)\ar^-{-\lten_ABe}[d]\ar[l]_{\cdot p_0}\\
{}_C(Be,\Theta^\ell\ten_ABe)\ar^-{q_\ell\cdot}[r]&{}_C(\A\lten_ABe,\underline{\Theta}^\ell\lten_ABe)&{}_C(\A\lten_ABe,\underline{\Theta}^\ell\lten_ABe)\ar[l]_{\cdot q_0}^\sim
}\]
The upper squares are clearly commutative. The lower squares are also commutative  since by definition $q_\ell=p_\ell\lten_A 1_{Be}$.
\end{proof}

\begin{lema}\label{3rd commutativity}
We have the following commutative diagram in $\underline{\CM}^{\ZZ}(C)$:
\[\xymatrix@C=5em{
\Theta^\ell\ten_ABe\ar^{q_\ell}[r]\ar^{\gamma_\ell}[d]&\underline{\Theta}^\ell\lten_ABe\ar^{\delta_\ell}[d]\\
Be(\ell)\ar^{q_0(\ell)}[r]&\A\lten_ABe(\ell)
}\]
\end{lema}

\begin{proof}
For the case $\ell=1$, the assertion is shown in Proposition~\ref{first step of induction}.
Assume that the assertion is true for $\ell-1$.
Consider the following commutative diagram:
\[\xymatrix@C=5em{
\Theta\ten_A\Theta^{\ell-1}\lten_ABe\ar^-{1_{\Theta}\ten_A\gamma_{\ell-1}}[d]\ar^-{1_{\Theta}\ten_Aq_{\ell-1}}[r]&
\Theta\ten_A\underline{\Theta}^{\ell-1}\lten_ABe\ar^{1_{\Theta}\ten_A\delta_{\ell-1}}[d]\ar[r]^-{p_0\ten_A(1_{\Theta\ten_A\underline{\Theta}^{\ell-1}\lten_ABe})}&
\underline{\Theta}\ten_{\A}\underline{\Theta}^{\ell-1}\lten_ABe\ar^{1_{\underline{\Theta}}\ten_{\A}\delta_{\ell-1}}[d]\\
\Theta\ten_ABe(\ell-1)\ar^-{1_{\Theta}\ten_Aq_0(\ell-1)}[r]\ar^{\gamma_1(\ell-1)}[d]&\Theta\ten_A\A\lten_ABe(\ell-1)\ar[r]^-{p_0\ten_A1_{\A\lten_ABe(\ell-1)}}&\underline{\Theta}\ten_{\A}\A\lten_ABe(\ell-1)\ar^{\delta_1(\ell-1)}[d]\\
Be(\ell)\ar^-{q_0(\ell)}[rr]&&\A\lten_ABe(\ell)
}\]
Clearly the upper right square is commutative.
The upper left square is commutative by our induction assumption, and the lower pentagon is commutative for the case $\ell=1$.
Thus the commutativity for the case $\ell$ follows from the biggest square.
\end{proof}

\begin{proof}[Proof of Proposition~\ref{big commutative diagram}]
We only have to show the statement for the case $m=0$.
For $\ell<0$, 
we have $\Hom_{\Dd(\A)}(\A,\underline{\Theta}^\ell)=0$ by $\gldim\A\le d-1$, and 
$\Hom_{\underline{\CM}^\ZZ(C)}(F(\A),F(\underline{\Theta}^\ell))\simeq 
\Hom_{\underline{\CM}^\ZZ(C)}(Be,Be(\ell))=\B_\ell=0$ by Proposition~\ref{cluster tilting graded CM}(b).
Hence $F_{\A,\underline{\Theta}^\ell}$ is an isomorphism in this case.

For $\ell\geq 0$ consider the following diagram:
\[\xymatrix@C=4em{
B_\ell\ar@{<-}^{{\rm \beta_\ell}}[r]_{\sim}\ar^{{\rm nat.}}[ddd]& H^0(\Theta^\ell)\ar[r]^{H^0(p_\ell)}\ar[d]_\wr&H^0(\underline{\Theta}^\ell)\ar_{\wr}[d]\\
&\Hom_{\Dd(A)}(A,\Theta^\ell)\ar^{-\lten_ABe}[d]&\Hom_{\Dd(\A)}(\A,\underline{\Theta}^\ell)\ar^{-\lten_ABe=F_{\A,\underline{\Theta}^\ell}}[d]\\
&\Hom_{\underline{\CM}^{\ZZ}(C)}(Be,\Theta^\ell\ten_ABe)\ar^-{q_\ell\ \cdot\ q_0^{-1}}[r]\ar^{\gamma_\ell\cdot}[d]&\Hom_{\underline{\CM}^{\ZZ}(C)}(\A\lten_ABe,\underline{\Theta}^\ell\lten_ABe)\ar^{\delta_\ell\cdot}[d]_{\wr}\\
\B_\ell\ar_-{\sim}[r]&\Hom_{\underline{\CM}^{\ZZ}(C)}(Be,Be(\ell))\ar_-{\sim}^-{q_0(\ell)\ \cdot\ q_0^{-1}}[r]&\Hom_{\underline{\CM}^{\ZZ}(C)}(\A\lten_ABe,\A\lten_ABe(\ell))
}\]
By Lemma~\ref{1st commutativity} the left hexagon is commutative, by Lemma~\ref{2nd commutativity} the upper right hexagon is commutative, and by Lemma~\ref{3rd commutativity} the lower square is commutative. Hence the whole diagram commutes.

Moreover by Lemma~\ref{commutativity B E} the map $\beta_\ell:H^0(\Theta^\ell)\simeq B_\ell$ induces  an isomorphism $ H^0(\underline{\Theta}^\ell)\simeq \B_\ell$.
Therefore the following diagram is commutative:
\[\xymatrix@C=4em{
H^0(\Theta^\ell)\ar[r]^{H^0(p_\ell)}\ar[dd]^{\beta_\ell}_\wr&
H^0(\underline{\Theta}^\ell)\ar[r]^{\sim}\ar[dd]_\wr
&\Hom_{\Dd(\A)}(\A,\underline{\Theta}^\ell)\ar[d]^{F_{\A,\underline{\Theta}^\ell}}\\
&&\Hom_{\underline{\CM}^\ZZ(C)}(F(\A),F(\underline{\Theta}^\ell))\ar[d]^{q_0(\ell)^{-1}\delta_\ell\ \cdot\ q_0}_\wr\\
B_\ell\ar[r]^{{\rm nat.}}&\B_\ell\ar[r]_(.4)\sim&\Hom_{\underline{\CM}^{\ZZ}(C)}(Be,Be(\ell))}
\]
 Thus $F_{\A,\underline{\Theta}^\ell}$ is an isomorphism.
\end{proof}

\begin{proof}[Proof of Theorem~\ref{main diagram}]

By Proposition~\ref{compatibility degree theta}, the functor $F$ restricted to the subcategory $\add\{\underline{\Theta}^\ell \mid \ell \in \ZZ\}\subset \Db(\A)$ induces a dense functor:
$$\add\{\underline{\Theta}^\ell \mid \ell \in \ZZ\}\to\add\{Be(\ell) \mid \ell\in \ZZ\}\subset\underline{\CM}^\ZZ(C).$$
This is an equivalence by Proposition~\ref{big commutative diagram}.
These subcategories are $(d-1)$-cluster tilting subcategories by Corollary~\ref{cor existence cluster tilting}(c) and Proposition~\ref{cluster tilting graded CM}(c). 
Thus $F$ is a triangle equivalence by Proposition~\ref{criterion for equivalence}. 

Since we have a commutative diagram
\[\xymatrix{
\Db(\A)/(-\ten_{\A}\underline{\Theta})\ar[r]^F\ar[d]^\pi&
(\underline{\CM}^{\ZZ}(C))/(1)\ar[d]^{{\rm nat.}}\\
\Cc_{d-1}(\A)\ar[r]^G&\underline{\CM}(C).
}\]
whose vertical functors are fully faithful and $F_{\A,\underline{\Theta}^\ell}$ is an isomorphism for any $\ell\in \ZZ$, we have that the map $G_{\pi\A,\pi\A}$ is an isomorphism.
Since $\pi\A\in\Cc_{d-1}(\A)$ and $G(\pi\A)=Be$ are $(d-1)$-cluster tilting objects by Corollary~\ref{cor existence cluster tilting}(b) and Theorem~\ref{cluster tilting CM}(d), we deduce that $G$ is a triangle equivalence again by Proposition~\ref{criterion for equivalence}.
\end{proof}

\section{Application to quotient singularities}\label{Application to quotient singularities}

In this section we apply the main theorem in the previous section
to invariant rings.

\subsection{Setup and main result}

Let $S$ be the polynomial ring $k[x_1,\ldots,x_d]$ over an algebraically closed field
$k$ of characteristic zero, and $G$ be a finite subgroup of $\SL_d(k)$ acting freely on $k^d\backslash\{0\}$.
The group $G$ acts on $S$ in a natural way. We
denote by $R:=S^G$ the invariant ring and by $S*G$ the skew group
algebra. Then $R$ is a Gorenstein isolated singularity of Krull
dimension $d$. We assume that $G$ is a cyclic group
generated by $g={\rm diag}(\zeta^{a_1},\ldots,\zeta^{a_d})$ with
a primitive $n$-th root $\zeta$ of unity and integers $a_j$ satisfying

\smallskip
\hspace{.5cm}(B1) $0<a_j<n$ and $(n,a_j)=1$ for any $j$ with $1\le j\le d$.
\smallskip

\hspace{.5cm}(B2) $a_1+\cdots+a_d=n$.
\smallskip

We regard $S=k[x_1,\cdots,x_d]$
as a $\frac{\ZZ}{n}$-graded ring $\bigoplus_{\ell\in\ZZ}S_{\frac{\ell}{n}}$
by putting $\deg x_j=\frac{a_j}{n}$. Since $G$ acts on $S$ by
$g\cdot x_i=\zeta^{a_i}x_i$, the invariant subring is given by
\[S^G=\bigoplus_{\ell\in\ZZ}S_{\ell}.\]
Now we define graded $S^G$-modules for each $i$ with $0\le i<n$ by
\[T^i:=\bigoplus_{\ell\in\ZZ}S_{\ell+\frac{i}{n}},\]
where the degree $\ell$ part of $T^i$ is
$S_{\ell+\frac{i}{n}}$. Then we have $T^0=S^G$. Let
\[T:=\bigoplus_{i=0}^{n-1}T^i\ \mbox{ and }\ T':=\bigoplus_{i=1}^{n-1}T^i.\]
Note that we have $T\simeq S$ as (ungraded) $S^G$-modules.
Define $k$-algebras by
\begin{eqnarray*}
A:=\End_{\Gr (S^G)}(T),&&\A:=\End_{\underline{\CM}^\ZZ(S^G)}(T)\\
B:=\End_{S^G}(T),&&\B:=\End_{\underline{\CM}(S^G)}(T).
\end{eqnarray*}
Then $B$ and $\B$ are graded algebras such that $A=B_0$ and $\A=\B_0$. 
We will give explicit presentations of $B$, $A$ and $\A$ in terms of quivers with relations in Proposition \ref{presentation}.

Let $e$ be the idempotent of $B=\End_{S^G}(T)$ associated with the direct summand $T^0$ of $T$. Then we have $eBe\simeq S^G$, $\A\simeq A/\langle e\rangle$ and $\B\simeq B/\langle e\rangle$.

Our main result in this section is the following.

\begin{thma}\label{main result for singularities}
Under the assumptions and notations above, we have the following.
\begin{itemize}
\item[(a)] The functor $\xymatrix{F:\Db(\A)\ar[r]^-{\rm Res.}&\Db(A)\ar[rr]^-{-\lten_A Be} && \Db(\gr S^G)\ar[r]& \underline{\CM}^{\ZZ}(S^G)}$ is a triangle equivalence.
Moreover $T\simeq Be$ is a tilting object in $\underline{\CM}^{\ZZ}(S^G)$.
\item[(b)] There exists a triangle equivalence $G:\Cc_{d-1}(\A)\to\underline{\CM}(S^G)$ 
making the diagram
$$\xymatrix{\Db(\A)\ar[d]^{\pi}\ar[rr]_\sim^{F}&& \underline{\CM}^{\ZZ}(S^G)\ar[d]^{\rm nat.} \\
\Cc_{d-1}(\A)\ar[rr]_\sim^{G} && \underline{\CM}(S^G)}$$
commutative, where $\Cc_{d-1}(\A)$ is the generalized $(d-1)$-cluster category of $\A$.
\end{itemize}
\end{thma}

As a consequence, we recover the following results.

\begin{cora}\label{invariant basic}
In the setup above, the following assertions hold.
\begin{itemize}
\item[(a)] \cite[III.1]{Aus78} The stable category $\underline{\CM}(S^G)$ of
maximal Cohen-Macaulay $R$-modules is a $(d-1)$-Calabi-Yau triangulated category.
\item[(b)] \cite[Thm 2.5]{Iya07a} The $S^G$-module $S$ is a
$(d-1)$-cluster tilting object in $\underline{\CM}(S^G)$. 
\end{itemize}
\end{cora}

As a special case of Theorem~\ref{main result for singularities} we have the following.

\begin{cora}
Let $G\subset\SL_3(k)$ be a finite cyclic subgroup satisfying (B1).
Then the stable category $\underline{\CM}(S^G)$ of
maximal Cohen-Macaulay modules is triangle equivalent to the
generalized 2-cluster category $\Cc_2(\underline{A})$ for a finite
dimensional algebra $\underline{A}$ of global dimension at most $2$.
\end{cora}

\begin{proof} 
We only have to check the condition (B2).
Let $g={\rm diag}(\zeta^{a_1}, \zeta^{a_2}, \zeta^{a_3})$ be a generator of $G$.
Since $0<a_i<n$ and $g\in\SL_3(k)$, we have $a_1+a_2+a_3=n$ or $2n$.
If this is $n$, then (B2) is satisfied. If this is $2n$, then
$g^{-1}={\rm diag}(\zeta^{n-a_1}, \zeta^{n-a_2}, \zeta^{n-a_3})$
satisfies (B2) since $(n-a_1)+(n-a_2)+(n-a_3)=n$.
\end{proof}

\begin{rema}
\begin{itemize}
\item[(a)] The triangle equivalence $F:\Db(\A)\to\underline{\CM}^{\ZZ}(S^G)$ is obtained by Ueda \cite{Ued08}.
Our proof is very different since he uses a strong theorem due to Orlov~\cite{Orl05}.
\item[(b)] The triangle equivalence $G:\Cc_{d-1}(\A)\to\underline{\CM}(S^G)$ is an analog of an independent result proved by Thanhoffer de V\"olcsey and Van den Bergh \cite[Proposition~1.3]{TV10}.  They use generalized cluster categories associated with quivers with potential
instead of those associated with algebras of finite global dimension. 
\end{itemize}
\end{rema}

\subsection{Proof of Theorem \ref{main result for singularities}}

Let $G$ be a finite cyclic subgroup of $\SL_d(k)$ generated by
$g={\rm diag}(\zeta^{a_1},\ldots,\zeta^{a_d})$ as above, and
let $S^G$, $B$, $\underline{B}$, $A$ and $\underline{A}$ be as
defined in the previous subsection. Then $B=\End_{S^G}(S)$  is isomorphic to the skew group algebra $S*G$ by \cite{Aus86,Yos90}, which is
known to have global dimension $d$. We want to show that conditions
(A1*) to (A4) in the previous section are satisfied in this case. We
start with condition (A1*), and here we need some notation.

First we give a concrete description of the McKay quiver $Q$ of the cyclic group $G$ \cite{Mck80}.
The set $Q_0$ of vertices is $\ZZ/n\ZZ$.
The arrows are
\[x_j^i=x_j:i\to i+a_j\quad (i\in\ZZ/n\ZZ,\ 1\le j\le d).\]

\begin{prop}\label{presentation}
\begin{itemize}
\item[(a)] A presentation of $B$ is given by the McKay quiver with commutative relations
\[x_{j'}^{i+a_j}\ x_j^i=x_{j}^{i+a_{j'}}\ x_{j'}^i\quad (i\in\ZZ/n\ZZ,\ 1\le j,j'\le d).\]
\item[(b)] A presentation of $A$ is obtained from that of $B$ by
removing all arrows $x_j^i:i\to i'$ with $i>i'$.
\item[(c)] A presentation of $\A$ is obtained from that of $A$ by removing the vertex $0$.
\end{itemize}
\end{prop}

\begin{proof}
(a) This is known (e.g.  \cite[Prop. 2.8(3)]{CMT07},\cite[Cor. 4.2]{BSW10}).

(b) By our grading on $T$, the degree of the morphism $x_j^i:T^i\to T^{i'}$ is $0$ if $i<i'$, and $1$ otherwise. Thus we have the assertion.

(c) This is clear.
\end{proof}

We denote by $Q_\ell$ the set of paths of length $\ell$, and by $kQ_\ell$ the $k$-vector space with basis $Q_\ell$.
Then $kQ_0$ is a $k$-algebra which we denote by $L:=kQ_0$.
Clearly we have
\[kQ_\ell=\underbrace{kQ_1\otimes_L\cdots\otimes_LkQ_1}_{\ell\textrm{ times}}.\]
Define a vector space $U_\ell$ as the factor space of $kQ_\ell$ modulo the subspace generated by
\[v\otimes x_i\otimes x_j\otimes v'+v\otimes x_j\otimes x_i\otimes v'.\]
We denote by $v_1\wedge v_2\wedge \cdots\wedge v_\ell$ the image of $v_1\otimes v_2\otimes \cdots\otimes v_\ell$ in $U_\ell$.
Then $U_\ell$ has a basis consisting of
\[x_{j_\ell}\wedge x_{j_{\ell-1}}\wedge \cdots\wedge x_{j_1}\]
where
\[\xymatrix{i\ar[r]^-{x_{j_1}}&i+a_{j_1}\ar[r]^-{x_{j_2}}&\cdots\ar[r]^-{x_{j_\ell}}&i+a_{j_1}+\cdots+a_{j_\ell}}\]
is a path of length $\ell$ satisfying $j_1<j_2<\cdots<j_\ell$. Now let
\[\xymatrix{P_\bullet:=(B\otimes_L U_d\otimes_LB\ar[r]^-{\delta_d}&B\otimes_L U_{d-1}\otimes_LB\ar[r]^-{\delta_{d-1}}&\cdots\ar[r]^-{\delta_1}&B\otimes_L U_0\otimes_LB),}\]
where $\delta_\ell$ is defined by
\begin{eqnarray*}
&&\delta_\ell(b\otimes(x_{j_1}\wedge x_{j_2}\wedge \cdots\wedge x_{j_{\ell-1}}\wedge x_{j_\ell})\otimes b')\\
&:=&\sum_{i=1}^\ell(-1)^{i-1}(bx_{j_i}\otimes(x_{j_1}\wedge\cdots\stackrel{\vee}{x}_{j_i}\cdots\wedge x_{j_\ell})\otimes b' + b\otimes(x_{j_1}\wedge\cdots\stackrel{\vee}{x}_{j_i}\cdots\wedge x_{j_\ell})\otimes x_{j_i}b').
\end{eqnarray*}
Then we have the following result which implies the condition (A1*).

\begin{thma}\label{S*G is CY}
The complex $P_\bullet$ is a projective resolution of the graded $B^{\rm
e}$-module $B$ satisfying $P_\bullet\simeq P^{\vee}_{\bullet}[d](-1)$ in $\Cc^{\rm b}(\grproj B^{\rm e})$.
In particular $B$ is a bimodule $d$-Calabi-Yau algebra of Gorenstein
parameter $1$.
\end{thma}

\begin{proof}
The assertion not involving the grading is known and easy to check (e.g. \cite[Thm 6.2]{BSW10}). 
We will show that each $\delta_{\ell}$ is homogeneous of degree $0$ by introducing a certain grading on $P_\bullet$.
Define the degree map $g:Q_1\to\ZZ$ by
\[g(i\xrightarrow{x_j}i'):=\left\{
\begin{array}{cc}
0& 0\le i<i'<n,\\
1& 0\le i'<i<n.
\end{array}\right.\]
Then we have a well-defined degree map
\[g(x_{j_1}\wedge \cdots\wedge x_{j_\ell}):=g(x_{j_1})+\cdots+g(x_{j_\ell})\]
on basis vectors of $U_\ell$. Since the value is always $0$ or $1$ by the condition (B2) $a_1+\cdots+a_d=n$, we have a decomposition
\[U_\ell=U_\ell^0\oplus U_\ell^1\]
where $U_\ell^0$ (respectively, $U_\ell^1$) is the subspace spanned by the elements of degree $0$ (respectively, $1$).
We regard $U_\ell^0$ as having degree $0$ and $U_\ell^1$ as having degree $1$. Then each map $\delta_\ell$ is homogeneous of degree $0$. 
\end{proof}

We proceed to show the other conditions.

\begin{lema}\label{check}
The graded algebra $S*G$ satisfies the conditions (A1*), (A2), (A3) and (A4) in Theorem
\ref{main diagram}.
\end{lema}

\begin{proof}
(A1*) This was shown in the previous theorem.

(A2)  The ring $B=S*G$ is clearly noetherian.

(A3) $S^G$ is an isolated singularity by (B1).
Then the stable category $\underline{\CM}(S^G)$ has finite dimensional homomorphism spaces \cite{Aus78,Yos90}. Hence $\dim_k\B$ is finite.

(A4) It is a direct consequence of the definition of $A$ that the vertex $0$ in the McKay quiver is a source.
We use the idempotent $e$ corresponding to this vertex.
\end{proof}

Now Theorem~\ref{main result for singularities} is an immediate
consequence of Theorem~\ref{main diagram} and
Lemma~\ref{check}. \qed

\medskip
In the subsections~\ref{subsectiond=2},~\ref{subsectiond=3} and~\ref{subsectiond}, which are devoted to examples, we use the notation $\frac{1}{n}(a_1,\ldots,a_d)$ for the element
${\rm diag}(\zeta^{a_1},\ldots,\zeta^{a_d})\in\SL _d(k)$, where
$a_1+\ldots+a_d=n$ and $\zeta$ is a  primitive $n$-root of unity.

\subsection{Example: Case $d=2$}\label{subsectiond=2}

Let $G\subset\SL_2(k)$ be a finite cyclic subgroup. Then there
exists a generator of the form $\frac{1}{n}(1,n-1)$. The algebra
$S*G$ is presented by the McKay quiver
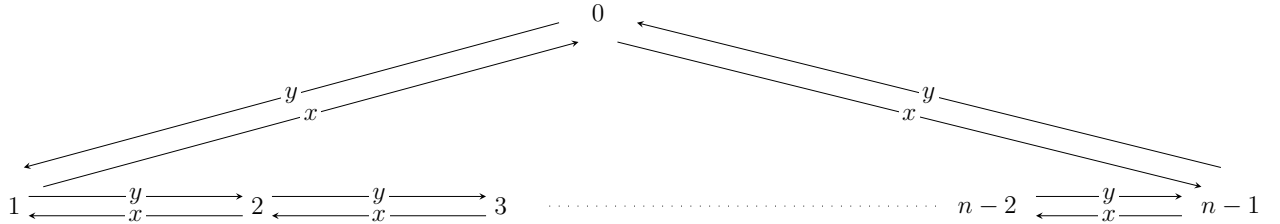
\[\scalebox{.8}{
\begin{tikzpicture}[>=stealth, scale=.8]
\node (P1) at (0,0){$1$}; \node (P2) at (5,0){$2$}; \node (P3) at
(10,0){$3$};

\node(P6) at (20,0){$n-2$}; \node(P7) at (25,0){$n-1$}; \node(P0) at
(12,4){$0$}; \draw [->] (0.3,0.2)--node[fill=white,inner
sep=.5mm,]{$y$}(4.7,0.2); \draw [<-]
(0.3,-0.2)--node[fill=white,inner sep=.5mm,]{$x$}(4.7,-0.2);

\draw [->] (5.3,0.2)--node[fill=white,inner
sep=.5mm,]{$y$}(9.7,0.2);
\draw [<-] (5.3,-0.2)--node[fill=white,inner
sep=.5mm,]{$x$}(9.7,-0.2); \draw[loosely dotted](11,0)--(19,0);

\draw [->] (21,0.2)--node[fill=white,inner sep=.5mm,]{$y$}(24,0.2);

\draw [<-] (21,-0.2)--node[fill=white,inner
sep=.5mm,]{$x$}(24,-0.2);

\draw[->] (11.2,3.8)--node[fill=white,inner
sep=.5mm,]{$y$}(0.2,0.8); \draw[<-]
(11.6,3.4)--node[fill=white,inner sep=.5mm,]{$x$}(0.6,0.4);

\draw[<-] (12.8,3.8)--node[fill=white,inner
sep=.5mm,]{$y$}(24.8,0.8); \draw[->]
(12.4,3.4)--node[fill=white,inner sep=.5mm,]{$x$}(24.4,0.4);
\end{tikzpicture}}
\]
with the commutativity relation $xy=yx$. The grading induced by the
generator $\frac{1}{n}(1,n-1)$ makes the arrows $x$ of degree $0$
and the arrows $y$ of degree $1$. The idempotent corresponding to the direct summand $T_0$ of $T$ corresponds to the vertex $0$ of the McKay quiver. Hence, the algebra $\A=\End_{\underline{\CM}^{\ZZ}(S^G)}(T)$ is isomorphic to $kQ$ where $Q$ is $A_{n-1}$ with the linear orientation.  Hence by Theorem~\ref{main result for singularities}, we obtain a triangle equivalence
$\underline{\CM}(S^G)\simeq \Cc_1(A_{n-1})$.

\bigskip
More generally, if $G$ is a finite subgroup (not necessarily cyclic) of $\SL_2(k)$, the algebra $B=S*G$ is Morita equivalent to the preprojective algebra $\Pi_{2}(\widetilde{Q})$ of an extended Dynkin quiver $\widetilde{Q}$. There exists a $\mathbb{Z}$-grading on $B$ such that $A:=B_0$ is Morita equivalent to the
path algebra $k\widetilde{Q}$ and $B$ is bimodule 2-Calabi-Yau of
Gorenstein parameter $1$. Moreover $B$ has an idempotent $e$ such that $eBe=S^G$ and $e$ is the exceptional vertex of $\widetilde{Q}$. Thus by Theorem \ref{main diagram} we have a triangle equivalence
$\Cc_{1}(kQ)\simeq\underline{\CM}(S^G)$ for
$Q:=\widetilde{Q}\backslash\{e\}$. 

Moreover, the category $\Cc_1(kQ)$ is equivalent to the category $\proj \Pi_2(kQ)$, where $\Pi_2(kQ)$ is the preprojective algebra associated to the Dynkin quiver $Q$.
Hence we recover the well-known proposition below.

\begin{prop}\label{2CY}
Let $G\subset\SL_2(k)$ be a finite subgroup and $Q$ be the corresponding Dynkin quiver.
\begin{itemize}
\item[(a)] \cite{Rei87,RV89,BSW10, Ami07} We have a triangle equivalence $\underline{\CM}(S^G)\simeq\Cc_1(kQ)$ and an equivalence $\underline{\CM}(S^G)\simeq\proj\Pi_2(kQ)$.

\item[(b)] \cite{KST07,LP06} We have a triangle equivalence $\underline{\CM}^{\ZZ}(S^G)\simeq\Db(kQ)$ and an equivalence $\underline{\CM}^{\ZZ}(S^G)\simeq\gr\proj\Pi_2(kQ)$.
\end{itemize}
\end{prop}

\begin{rema}\label{in fact triangle}
From \cite{Rei87,RV89,BSW10}, we get an equivalence $\underline{\CM}(S^G)\simeq\Cc_1(kQ)$. This equivalence implies that the category $\underline{\CM}(S^G)$ is standard, that is, is equivalent to the mesh category of its Auslander-Reiten quiver. Since it is also an algebraic triangulated category, one deduces that it is a triangle equivalence by \cite[Theorem 7.2]{Ami07}. It was also proved in \cite[Corollary 9.3]{Ami07} that the category $\proj \Pi_2(kQ)$ is naturally triangulated.
\end{rema}

\begin{rema} Let $\A$ be a finite-dimensional algebra of global dimension at most $1$. Then, if $k$ is algebraically closed, $\A$ is Morita equivalent to the path algebra $kQ$ of an acyclic quiver $Q$. 
The 1-cluster category $\Cc_1(kQ)$ is $\Hom$-finite  if and only if $Q$ is of Dynkin type. Thus we obtain a kind of converse of Theorem \ref{main diagram} for $d=2$: every $1$-cluster category can be realized as the stable category of Cohen-Macaulay modules over an isolated singularity.
\end{rema}

\subsection{Example: Case $d=3$}\label{subsectiond=3} $ $
Let $G\subset{\sf SL}_3(k)$ be the subgroup generated by $\frac{1}{5}(1,2,2)$. Then $B=S*G$ is presented by the McKay quiver
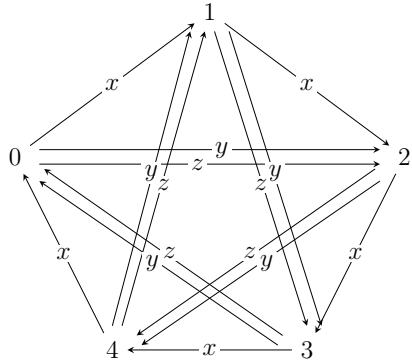
\[\scalebox{.8}{
\begin{tikzpicture}[>=stealth, scale=.8]
\node (P0) at (0,0){$0$}; \node (P1) at (4,3){$1$}; \node (P2) at
(8,0){$2$}; \node (P3) at (6,-4){$3$}; \node (P4) at (2,-4){$4$};

\draw [->] (0.5,0.15) -- node [fill=white,inner
sep=.5mm,xshift=2mm]{$y$} (7.5,0.15); \draw [->] (0.5,-0.15) -- node
[fill=white,inner sep=.5mm, xshift=-2mm]{$z$} (7.5,-0.15); \draw
[->] (7.4,-0.25) -- node [fill=white,inner sep=.5mm,
xshift=-1mm]{$z$} (2.5,-3.7); \draw [->] (7.5,-.5) -- node
[fill=white,inner sep=.5mm, xshift=1mm]{$y$} (2.6,-3.9); \draw [->]
(2,-3.5) -- node [fill=white,inner sep=.5mm, yshift=1mm]{$y$}
(3.6,2.6); \draw [->] (2.2,-3.5) -- node [fill=white,inner sep=.5mm,
yshift=-1mm]{$z$} (3.9,2.6); \draw [->] (4.4,2.6) -- node
[fill=white,inner sep=.5mm, yshift=1mm]{$y$} (6.3,-3.5); \draw [->]
(4.1,2.6) -- node [fill=white,inner sep=.5mm, yshift=-1mm]{$z$}
(6,-3.5); \draw [<-] (0.6,-0.25) -- node [fill=white,inner sep=.5mm,
xshift=1mm]{$z$} (5.5,-3.7); \draw [<-] (0.5,-.5) -- node
[fill=white,inner sep=.5mm, xshift=-1mm]{$y$} (5.4,-3.9);

\draw [->] (P0)-- node [fill=white,inner sep=.5mm]{$x$}(P1); \draw
[->] (P1)-- node [fill=white,inner sep=.5mm]{$x$}(P2); \draw [->]
(P2)-- node [fill=white,inner sep=.5mm]{$x$}(P3); \draw [->] (P3)--
node [fill=white,inner sep=.5mm]{$x$}(P4); \draw [->] (P4)-- node
[fill=white,inner sep=.5mm]{$x$}(P0);
\end{tikzpicture}}
\]
with the commutativity relations $xy=yx$, $yz=zy$, $zx=xz$. By the
choice of the grading, the algebra $A$, which is the degree 0 part
of $B$, is presented by the quiver
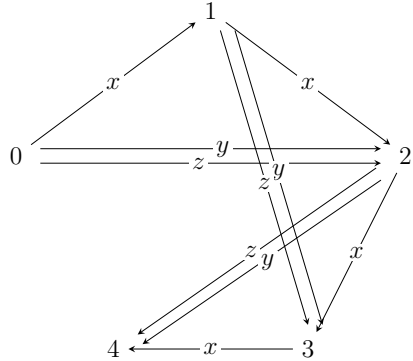
\[\scalebox{.8}{
\begin{tikzpicture}[>=stealth, scale=.8]
\node (P0) at (0,0){$0$}; \node (P1) at (4,3){$1$}; \node (P2) at
(8,0){$2$}; \node (P3) at (6,-4){$3$}; \node (P4) at (2,-4){$4$};

\draw [->] (0.5,0.15) -- node [fill=white,inner
sep=.5mm,xshift=2mm]{$y$} (7.5,0.15); \draw [->] (0.5,-0.15) -- node
[fill=white,inner sep=.5mm, xshift=-2mm]{$z$} (7.5,-0.15); \draw
[->] (7.4,-0.25) -- node [fill=white,inner sep=.5mm,
xshift=-1mm]{$z$} (2.5,-3.7); \draw [->] (7.5,-.5) -- node
[fill=white,inner sep=.5mm, xshift=1mm]{$y$} (2.6,-3.9);

\draw [->] (4.5,2.6) -- node [fill=white,inner sep=.5mm,
yshift=1mm]{$y$} (6.3,-3.5); \draw [->] (4.2,2.6) -- node
[fill=white,inner sep=.5mm, yshift=-1mm]{$z$} (6,-3.5);

\draw [->] (P0)-- node [fill=white,inner sep=.5mm]{$x$}(P1); \draw
[->] (P1)-- node [fill=white,inner sep=.5mm]{$x$}(P2); \draw [->]
(P2)-- node [fill=white,inner sep=.5mm]{$x$}(P3); \draw [->] (P3)--
node [fill=white,inner sep=.5mm]{$x$}(P4);

\end{tikzpicture}}
\]
with the commutativity relations. The idempotent $e$ of the algebra $B$ corresponds to the summand $S^G$ which corresponds to the vertex $0$. Therefore $\underline{A}$ is presented by the quiver
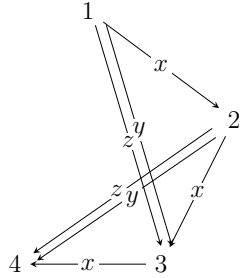
\[\scalebox{.8}{
\begin{tikzpicture}[>=stealth, scale=.6]

\node (P1) at (4,3){$1$}; \node (P2) at (8,0){$2$}; \node (P3) at
(6,-4){$3$}; \node (P4) at (2,-4){$4$};

\draw [->] (7.4,-0.25) -- node [fill=white,inner sep=.5mm,
xshift=-1mm]{$z$} (2.5,-3.7); \draw [->] (7.5,-.5) -- node
[fill=white,inner sep=.5mm, xshift=1mm]{$y$} (2.6,-3.9);

\draw [->] (4.5,2.6) -- node [fill=white,inner sep=.5mm,
yshift=1mm]{$y$} (6.3,-3.5); \draw [->] (4.2,2.6) -- node
[fill=white,inner sep=.5mm, yshift=-1mm]{$z$} (6,-3.5);

\draw [->] (P1)-- node [fill=white,inner sep=.5mm]{$x$}(P2); \draw
[->] (P2)-- node [fill=white,inner sep=.5mm]{$x$}(P3); \draw [->]
(P3)-- node [fill=white,inner sep=.5mm]{$x$}(P4);

\end{tikzpicture}}\]
with the commutativity relations. By Theorem~\ref{main result for singularities} the category $\underline{\CM}(S^G)$ is triangle equivalent to the generalized cluster category $\Cc_2(\underline{A})$ .

Now take another generator of the group $G$  given by
$\frac{1}{5}(3,1,1)$. Then the algebra $B$ is same as the above, but has a different grading. We denote by $A'$ its degree zero subalgebra.
One then easily checks that the algebra $\underline{A'}$ is given in this case by the quiver 
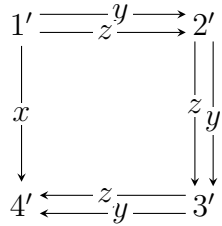
\[\scalebox{1}{
\begin{tikzpicture}[>=stealth, scale=1.2]

\node (P1) at (0,2){$1'$}; \node (P2) at (2,2){$2'$}; \node (P3) at
(2,0){$3'$}; \node (P4) at (0,0){$4'$};

\draw [->] (0.2,2.1)--node[fill=white,inner sep=.5mm, xshift=1mm]{$y$}(1.8,2.1); \draw [->]
(0.2,1.9)--node[fill=white,inner sep=.5mm, xshift=-1mm]{$z$}(1.8,1.9); \draw [->]
(2.1,1.8)--node[fill=white,inner sep=.5mm, 
yshift=-1mm]{$y$}(2.1,0.2); \draw [->]
(1.9,1.8)--node[fill=white,inner sep=.5mm,
yshift=1mm]{$z$}(1.9,0.2); \draw [->]
(1.8,-0.1)--node[fill=white,inner sep=.5mm,
xshift=1mm]{$y$}(0.2,-0.1); \draw [->]
(1.8,0.1)--node[fill=white,inner sep=.5mm, xshift=-1mm]{$z$}(0.2,
0.1);

\draw [->] (P1)-- node [fill=white,inner sep=.5mm]{$x$}(P4);

\end{tikzpicture}}
\]with commutativity relations.
 
By Theorem~\ref{main result for singularities} the category $\underline{\CM}(S^G)$ is triangle equivalent to the generalized cluster category $\Cc_2(\underline{A'})$. Hence we get a triangle equivalence between the generalized cluster categories $\Cc_2(\underline{A})\simeq\Cc_2(\underline{A'})$, that is,the algebras $\underline{A}$ and $\underline{A'}$ are \emph{cluster equivalent} in the sense of ~\cite{AO10}. However, one can show that the algebras $\underline{A}$ and $\underline{A'}$ are not derived equivalent since they have different Coxeter polynomials. (One can also see this using results of~\cite{AO10}.)
Now we have two different gradings on $S^G$, which we denote by $\ZZ$ and $\ZZ'$.
Then we have
$$\underline{\CM}^{\ZZ}(S^G)\simeq\Db(\A)\ {\not\simeq}\ \Db(\A')\simeq\underline{\CM}^{\ZZ'}(S^G).$$

\subsection{Example: General $d$}\label{subsectiond}
 Now let $d=n$ and $G$ be generated by $\frac{1}{d}(1,\ldots,1)$.  Then,
proceeding as before, it is not hard to see that $B=S*G$ is presented by the McKay quiver
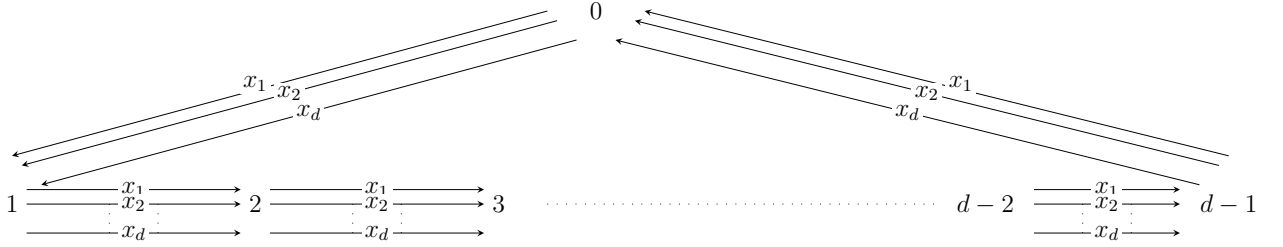
\[\scalebox{.8}{
\begin{tikzpicture}[>=stealth, scale=.8]
\node (P1) at (0,0){$1$}; \node (P2) at (5,0){$2$}; \node (P3) at
(10,0){$3$};

\node(P6) at (20,0){$d-2$}; \node(P7) at (25,0){$d-1$}; \node(P0) at
(12,4){$0$}; \draw [->] (0.3,0.3)--node[fill=white,inner
sep=.5mm,]{$x_1$}(4.7,0.3);
\draw [->] (0.3,0)--node[fill=white,inner sep=.5mm,]{$x_2$}(4.7,0);
\draw [->] (0.3,-0.6)--node[fill=white,inner
sep=.5mm,]{$x_d$}(4.7,-0.6);

\draw [->] (5.3,0.3)--node[fill=white,inner
sep=.5mm,]{$x_1$}(9.7,0.3);
\draw [->] (5.3,0)--node[fill=white,inner sep=.5mm,]{$x_2$}(9.7,0);
\draw [->] (5.3,-0.6)--node[fill=white,inner
sep=.5mm,]{$x_d$}(9.7,-0.6); \draw[loosely dotted](11,0)--(19,0);
\draw[loosely dotted] (2,0)--(2,-0.6); \draw[loosely dotted]
(3,0)--(3,-0.6); \draw[loosely dotted] (7,0)--(7,-0.6);
\draw[loosely dotted] (8,0)--(8,-0.6); \draw[loosely dotted]
(22,0)--(22,-0.6); \draw[loosely dotted] (23,0)--(23,-0.6); \draw
[->] (21,0.3)--node[fill=white,inner sep=.5mm,]{$x_1$}(24,0.3);
\draw [->] (21,0)--node[fill=white,inner sep=.5mm,]{$x_2$}(24,0);
\draw [->] (21,-0.6)--node[fill=white,inner
sep=.5mm,]{$x_d$}(24,-0.6);

\draw[->] (11,4)--node[fill=white,inner
sep=.5mm,xshift=-4mm]{$x_1$}(0,1); \draw[->]
(11.2,3.8)--node[fill=white,inner sep=.5mm,]{$x_2$}(0.2,0.8);
\draw[->] (11.6,3.4)--node[fill=white,inner
sep=.5mm,]{$x_d$}(0.6,0.4);

\draw[<-] (13,4)--node[fill=white,inner
sep=.5mm,xshift=4mm]{$x_1$}(25,1); \draw[<-]
(12.8,3.8)--node[fill=white,inner sep=.5mm,]{$x_2$}(24.8,0.8);
\draw[<-] (12.4,3.4)--node[fill=white,inner
sep=.5mm,]{$x_d$}(24.4,0.4);

\end{tikzpicture}}
\]
with the commutative relations $x_jx_i=x_jx_i$. Then, with the
grading corresponding to the generator $\frac{1}{d}(1,\ldots,1)$,
one can check that the algebra $A$ is the $d$-Beilinson algebra and the algebra $\underline{A}$ is given by the
quiver
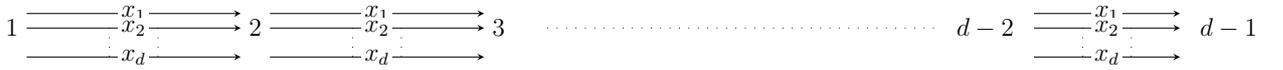
\[\scalebox{.8}{
\begin{tikzpicture}[>=stealth, scale=.8]
\node (P1) at (0,0){$1$}; \node (P2) at (5,0){$2$}; \node (P3) at
(10,0){$3$};

\node(P6) at (20,0){$d-2$}; \node(P7) at (25,0){$d-1$};

\draw [->] (0.3,0.3)--node[fill=white,inner
sep=.5mm,]{$x_1$}(4.7,0.3);
\draw [->] (0.3,0)--node[fill=white,inner sep=.5mm,]{$x_2$}(4.7,0);
\draw [->] (0.3,-0.6)--node[fill=white,inner
sep=.5mm,]{$x_d$}(4.7,-0.6);

\draw [->] (5.3,0.3)--node[fill=white,inner
sep=.5mm,]{$x_1$}(9.7,0.3);
\draw [->] (5.3,0)--node[fill=white,inner sep=.5mm,]{$x_2$}(9.7,0);
\draw [->] (5.3,-0.6)--node[fill=white,inner
sep=.5mm,]{$x_d$}(9.7,-0.6); \draw[loosely dotted](11,0)--(19,0);
\draw[loosely dotted] (2,0)--(2,-0.6); \draw[loosely dotted]
(3,0)--(3,-0.6); \draw[loosely dotted] (7,0)--(7,-0.6);
\draw[loosely dotted] (8,0)--(8,-0.6); \draw[loosely dotted]
(22,0)--(22,-0.6); \draw[loosely dotted] (23,0)--(23,-0.6); \draw
[->] (21,0.3)--node[fill=white,inner sep=.5mm,]{$x_1$}(24,0.3);
\draw [->] (21,0)--node[fill=white,inner sep=.5mm,]{$x_2$}(24,0);
\draw [->] (21,-0.6)--node[fill=white,inner
sep=.5mm,]{$x_d$}(24,-0.6);
\end{tikzpicture}}
\]
with the commutativity relations.

For the case $d=3$ the triangle equivalence $\Cc_2(\underline{A})\simeq \underline{\CM}(S^G)$ was already proved in \cite{KR08} using a recognition theorem for the acyclic $2$-cluster category.

\section{Examples coming from dimer models}

In this section we show that our main theorem applies to examples coming from dimer models which do not come from quotient singularities. This builds upon results from \cite{ Bro08, IU09, Dav11,Boc11} which we recall.

\subsection{3-Calabi-Yau algebras from dimer models}

Let $\Gamma$ be a bipartite graph on a torus. We denote by $\Gamma_0$ (resp. $\Gamma_1$, and $\Gamma_2$) the set of vertices (resp. edges and faces) of the graph. To such a graph we associate a quiver with a potential $(Q,W)$ in the sense of \cite{DWZ08}. The quiver viewed as an oriented graph on the torus is the dual of the graph $\Gamma$. Faces of $Q$ dual to white vertices are oriented clockwise and faces of $Q$ dual to black vertices are oriented anti-clockwise. Hence any vertex $v\in\Gamma_0$ corresponds canonically to a cycle $c_v$ of $Q$. The potential $W$ is defined as $$W=\sum_{v \textrm{ white}}c_v-\sum_{v \textrm{ black}}c_v.$$ Assume that there exists a consistent charge on this graph, that is, a map $R:Q_1\to \mathbb{R}_{>0}$ such that 
\begin{itemize}
\item $\forall v\in \Gamma_0\quad  \sum_{a\in Q_1, a\in c_v} R(a)=2.$
\item $\forall i\in Q_0\quad \sum_{a\in Q_1, s(a)=i}(1-R(a))+\sum_{a\in Q_1, t(a)=i} (1-R(a))=2$
\end{itemize}
 
 For such a consistent dimer model, there always exists a perfect matching, that is a subset $D$ of $\Gamma_1$ such that any vertex of $\Gamma_0$ belongs to exactly one edge in $D$. Since $Q$ is the dual of $\Gamma$ we regard $D$  as a subset of $Q_1$. We define a grading $d_D$ on $kQ$ as follows: $$d_D(a)=\left\{\begin{array}{ll}1 & \textrm{if } a\in D\\ 0 & \textrm{else.}\end{array}\right.$$ Since $D$ is a perfect matching, for any vertex $v\in \Gamma_0$ the cycle $c_v$ contains exactly one arrow of degree $1$, and then the potential $W$ is homogeneous of degree $1$. Hence $D$ induces a grading $d_D$ on the Jacobian algebra $B$. In other words $D$ is a cut of $(Q,W)$ in the sense of \cite{HI11}.
 
 \begin{prop}\label{dimer 3CY Gor1}
 Let $B$ be a Jacobian algebra coming from a consistent dimer model. Any perfect matching induces a grading on $B$ making it bimodule $3$-Calabi-Yau of Gorenstein parameter $1$.
 \end{prop}
 
 \begin{proof}
 We define the following complex $P_\bullet$ of graded projective $B$-bimodules:
 $$\xymatrix{\cdots\ar[r] & 0\ar[r] &P_3\ar[r]^-{\partial_2} & P_2\ar[r]^-{\partial_1} & P_1\ar[r]^-{\partial_0} & P_0\ar[r] & 0 \ar[r] & \cdots},$$
where $$\begin{array}{rcl} P_0 &= &\bigoplus_{i\in Q_0} Be_i\ten e_i B\\ P_1&=&\bigoplus_{a\in Q_1} (Be_{t(a)}\ten e_{s(a)}B)(-d(a))\\P_2&=&\bigoplus_{b\in Q_1} (Be_{s(b)}\ten e_{t(b)}B)(1-d(b))\\P_3&=&\bigoplus_{i\in Q_0}Be_i\ten e_i B(-1)\end{array}$$ and where the maps $\partial_0$, $\partial_1$ and $\partial_2$ are defined as follows.
\[ \begin{array}{rcl} \partial_2(e_i\ten e_i) & = & \sum_{a,t(a)=i}a\ten e_i -\sum_{b, s(b)=i}e_i\ten
b; \\
\partial_1(e_{s(b)}\ten e_{t(b)}) & = & \sum_{a\in Q_1}\partial_{b,a}W  \textrm{ where }\partial _{b,a}(bpaq)=p\ten
q\in Be_{t(a)}\ten e_{s(a)}B  ;\\
\partial_0 (e_{t(a)}\ten e _{s(a)}) & = & a\ten e_{s(a)} -e_{t(a)}\ten a.  \end{array}\]

By \cite[Thm 7.7]{Bro08} this complex is a projective resolution of $B$ as a bimodule and satisfies $P^\vee_\bullet\simeq P[3]$ in $\Cc^{\rm b}(\proj B^{\rm e})$. It is then easy to check that, as a graded complex, it satisfies $P^\vee_\bullet\simeq P[3](-1)$ in $\Cc^{\rm b}(\gr\proj B^{\rm e})$. Hence the graded algebra $B$ is bimodule $3$-Calabi-Yau of Gorenstein parameter 1. 
\end{proof}

 \begin{rema}
 It is proved in \cite{Bro08,Dav11,Boc11} that the Jacobian algebra $B={\rm Jac}(Q,W)$ is a non-commutative crepant resolution of its center $C=Z(B)$ which is the coordinate ring of a Gorenstein affine toric threefold. Moreover the coordinate ring of any Gorenstein affine toric threefold can be obtained from a consistent dimer model \cite{Gul08,IU09}.
\end{rema}
The following result gives an interpretation of the stable category of Cohen-Macaulay modules over certain Gorenstein affine toric threefold in terms of cluster categories.

\begin{thma}\label{toric thm}
Let $\Gamma$ be a consistent dimer model, and denote by $B={\rm Jac}(Q,W)$ the associated Jacobian algebra. Assume there exists a perfect matching $D$ and a vertex $i$ of $Q$ with the following properties:
\begin{itemize}
\item the degree zero part $A$ of $B$ with respect to $d_D$ is finite dimensional
\item $i$ is a source of the quiver $Q-D$.
\item the algebra $B/\langle e_i\rangle$ is finite dimensional.
\end{itemize}
Denote by $C$ the center of the algebra $B$, and $\underline{A}$ the algebra $A/\langle e_i \rangle$. Then the algebra $C$ is a Gorenstein isolated singularity, and we have the following triangle equivalences
$$\xymatrix{\Db( \underline{A})\ar[rr]^-\sim\ar[d] && \underline{\CM}^\mathbb{Z}(C)\ar[d]\\ \Cc_2(\underline{A})\ar[rr]^-\sim && \underline{\CM}(C)}$$
where $\Cc_2(\underline{A})$ is the generalized 2-cluster category associated to the algebra $\underline{A}$.
\end{thma}

\begin{proof}
The algebra $B$ satisfies \eqref{(A1*)} by Proposition \ref{dimer 3CY Gor1}. The algebra $C$ is a Gorenstein affine toric threefold and $B$ is a finitely generated $C$-module, hence $B$ is noetherian. The hypothesis on the perfect matching $D$ and the vertex $i$ are clearly equivalent to (A3) and (A4).  Moreover by \cite[Lemma 5.6]{Bro08}, the center of $B$ is isomorphic to $eBe$ for any primitive idempotent $e$ of $B$. Hence Theorem \ref{toric thm} is a consequence of Theorem \ref{main diagram}.
\end{proof}

\subsection{Examples}

Let $\Gamma$ and $D$ be given by the following picture.

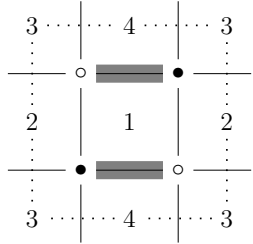
\[\scalebox{.8}{
\begin{tikzpicture}[>=stealth, scale=.8]
\node (P1) at (1,1){$\bullet$}; \node (P2) at (3,1){$\circ$}; \node (P3) at
(3,3){$\bullet$}; \node (P4) at (1,3){$\circ$};
\node (Q1) at (0,0){$3$};\node (M1) at (0,2){$2$}; \node (M2) at (2,4){$4$};\node(Q2) at (0,4){$3$};\node(Q3) at (4,4){$3$};\node(M3) at (4,2){$2$};\node(Q4) at (4,0){$3$};\node (M4) at (2,0){$4$};
\draw[loosely dotted,thick]  (Q1)--(M1)--(Q2)--(M2)--(Q3)--(M3)--(Q4)--(M4)--(Q1);
\draw[line width=0.3cm,color=gray] (P1)--(P2);\draw[line width=0.3cm,color=gray] (P3)--(P4);
\draw (P1)--(P2)--(P3)--(P4)--(P1);
\draw (-0.5,1)--(P1)--(1,-0.5); \draw (4.5,1)--(P2)--(3,-0.5);\draw (-0.5,3)--(P4)--(1,4.5);\draw (3,4.5)--(P3)--(4.5,3);
\node at (2,2){$1$};
\end{tikzpicture}}\]

The associated Jacobian algebra $B$ is presented by the quiver

\[\scalebox{.8}{
\begin{tikzpicture}[>=stealth, scale=.8]
\node (P1) at (0,0){$1$};\node (P2) at (3,0){$2$};\node (P3) at (3,3){$3$};\node(P4) at (0,3){$4$};
\draw[->] (0.3,0.1)--node [swap,yshift=2mm] {$x_1$}(2.7,0.1);\draw[->] (0.3,-0.1)--node [swap,yshift=-2mm] {$x_2$}(2.7,-0.1);
\draw[<-] (0.3,3.1)--node [swap,yshift=2mm] {$z_2$}(2.7,3.1);\draw[<-] (0.3,2.9)--node [swap,yshift=-2mm] {$z_1$}(2.7,2.9);
\draw[->] (2.9,0.3)--node [swap,xshift=-2mm] {$y_1$}(2.9,2.7);\draw[->] (3.1,0.3)--node [swap,xshift=2.5mm] {$y_2$}(3.1,2.7);
\draw[<-] (-0.1,0.3)--node [swap,xshift=-3mm] {$w_2$}(-0.1,2.7);\draw[<-] (0.1,0.3)--node [swap,xshift=3mm] {$w_1$}(0.1,2.7);
\node (P5) at (12,1){with potential $W=w_1z_1y_1x_1+w_2z_2y_2x_2-w_1z_2y_1x_2-w_2z_1y_2x_1$.};
\end{tikzpicture}}\]

The center $C$ of this algebra is the semigroup algebra $C=\mathbb{C}[\mathbb{Z}^3\cap \sigma^\vee]$ where $\sigma^\vee$ is the positive cone 
$$\sigma^\vee=\{\lambda_1 n_1+\lambda_2 n_2+\lambda_3 n_3+\lambda_4 n_4, \lambda_i\geq 0\},\quad n_1=\begin{bmatrix}1\\1\\1\end{bmatrix}, n_2=\begin{bmatrix}1\\-1\\1\end{bmatrix},n_3=\begin{bmatrix}-1\\1\\1\end{bmatrix},n_4=\begin{bmatrix}-1\\-1\\1\end{bmatrix}.$$
The algebra $C$ is a homogenous coordinate algebra of $\mathbb{P}^1\times\mathbb{P}^1$.

Then the perfect matching $D$ corresponds to $\{x_1, x_2\}$. Thus the quiver of the subalgebra $A=B_0$ is acyclic so $A$ is finite dimensional and the vertex $1$ becomes a source in the quiver of $A$. Moreover, the algebra $\underline{B}=B/\langle e_1\rangle $ is the path algebra of an acyclic quiver, so it is finite dimensional. Therefore we can apply Theorem \ref{toric thm} and we obtain a triangle equivalence
$\Cc_2(\underline{A})\simeq\underline{\CM}(C)$ where $\underline{A}$ is the path algebra of the quiver$\xymatrix{2\ar@<.5ex>[r] \ar@<-.5ex>[r]&3\ar@<.5ex>[r] \ar@<-.5ex>[r] & 4}$.

\medskip

We end this paper by giving a non-commutative example. Note that in Theorem \ref{main diagram} the algebra $C$ is not necessarily commutative, and the idempotent $e$ is not necessarily primitive.

Let $\Gamma$ be the following  dimer model.

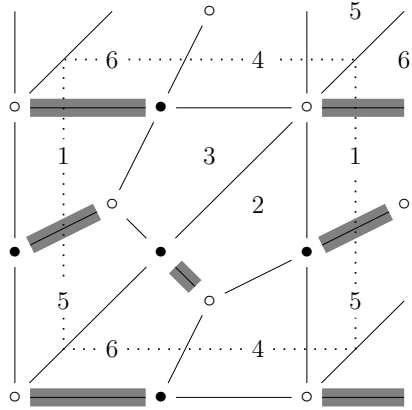
\[\scalebox{.8}{\begin{tikzpicture}[scale=.8]
\node(P1) at (0,0){$\circ$};\node(P2) at (3,0){$\bullet$};\node (P3) at (6,0){$\circ$};
\node(P4) at (4,2){$\circ$};\node(P5) at (0,3){$\bullet$};\node (P6) at (3,3){$\bullet$};\node (P7) at (6,3){$\bullet$};
\node (P8) at (2,4){$\circ$}; \node(P9) at (8,4){$\circ$};\node (P10) at (0,6){$\circ$}; \node (P11) at (3,6){$\bullet$};\node (P12) at (6,6){$\circ$}; \node (P13) at (4,8){$\circ$}; 

\draw [line width=0.3cm,color=gray](P5)--(P8);\draw [line width=0.3cm,color=gray](P10)--(P11);\draw [line width=0.3cm,color=gray](P6)--(P4);\draw [line width=0.3cm,color=gray](P7)--(P9);\draw [line width=0.3cm,color=gray](P12)--(8,6);\draw [line width=0.3cm,color=gray](P1)--(P2);\draw [line width=0.3cm,color=gray](P3)--(8,0);

\draw (P1)--(P2)--(P4)--(P6)--(P8)--(P5)--(P1)--(P6)--(P12)--(8,8);\draw(P4)--(P7)--(P12)--(P11)--(P8);
\draw (P2)--(P3)--(P7)--(P9);\draw (P5)--(P10)--(P11)--(P13);\draw (P10)--(2,8);\draw (8,6)--(P12)--(6,8);\draw(P3)--(8,2);\draw(P3)--(8,0); \draw (P10)--(0,8);

\node (Q1) at (2,1){$6$};\node (Q2) at (5,1){$4$};\node (Q3) at (1,2){$5$}; \node (Q4) at (7,2){$5$};\node (Q5) at (1,5){$1$};\node (Q6) at (7,5){$1$};\node (Q7) at (2,7){$6$};\node (Q8) at (5,7){$4$};
\draw [loosely dotted, thick] (Q3)--(1,1)--(Q1)--(Q2)--(7,1)--(Q4)--(Q6)--(7,7)--(Q8)--(Q7)--(1,7)--(Q5)--(Q3);
\node at(4,5){$3$};\node at (5,4){$2$};\node at (7,8){$5$};\node at (8,7){$6$};

\end{tikzpicture}}\]

The associated Jacobian algebra $B$ is presented by the quiver 

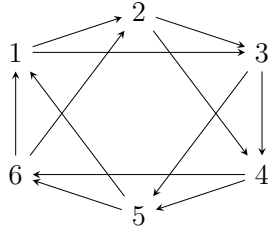
\[\scalebox{.9}{
\begin{tikzpicture}[>=stealth, scale=.6]
\node (P1) at (0,4){$1$};\node (P2) at (3,5){$2$};\node (P3) at (6,4){$3$};\node (P4) at (6,1){$4$};\node (P5) at (3,0){$5$};\node (P6) at (0,1){$6$};
\draw[->] (P1)--(P2);\draw[->](P2)--(P3);\draw[->](P3)--(P4);\draw[->] (P4)--(P5);\draw[->] (P5)--(P6);\draw[->](P6)--(P1);
\draw[->] (P1)--(P3);\draw[->] (P3)--(P5);\draw[->] (P5)--(P1);\draw[->] (P2)--(P4);\draw[->] (P4)--(P6);\draw[->] (P6)--(P2);
\node at (9,3){with potential}; 
\node at (14,2){$W=a_{65}a_{54}a_{43}a_{32}a_{21}a_{16}+a_{26}a_{64}a_{42}+a_{15}a_{53}a_{31}$ };
\node at (16,1) {$-a_{16}a_{64}a_{43}a_{31}-a_{65}a_{53}a_{32}a_{26}-a_{21}a_{15}a_{54}a_{42}$.};
\end{tikzpicture}}\]

In this case it is easy to check that the algebra $B/\langle e\rangle$ is not finite dimensional for any primitive idempotent $e$, or in other words, the center of $B$ is not an isolated singularity. However, the degree zero subalgebra $A=B_0$ and the algebra $\underline{B}=B/\langle e_1+e_2 \rangle$ are the path algebras of acyclic quivers, and are therefore finite dimensional. We can apply Theorem \ref{main diagram} with the perfect matching $D$ described in the picture above. We obtain a triangle equivalence $\Cc_2(\underline{A})\simeq \underline{\CM}(C)$ where $\underline{A}$ is the path algebra of the quiver 

\[\scalebox{.8}{
\begin{tikzpicture}[>=stealth, scale=.6]
\node (P3) at (6,4){$3$};\node (P4) at (6,1){$4$};\node (P5) at (3,0){$5$};\node (P6) at (0,1){$6$};
\draw[->](P3)--(P4);\draw[->] (P4)--(P5);\draw[->] (P5)--(P6);\draw[->] (P3)--(P5);\draw[->] (P4)--(P6);
\end{tikzpicture}}\]

\end{document}